\documentclass[a4paper]{article}

\usepackage{amssymb,amsmath,amsthm}	
\usepackage{a4wide}
\usepackage{graphicx}      										
\usepackage{tikz}          										
\usepackage[numbers]{natbib}
\usepackage{dsfont}
\usepackage{amsmath}
\usepackage{xspace}
\usepackage[normalem]{ulem}
\usepackage{doi}
\usepackage{BOONDOX-cal}

\newcommand{\diag}{\mathop{\mathrm{diag}}}
\newcommand{\bs}{\ensuremath{\boldsymbol}}
\DeclareMathOperator*{\argmax}{argmax}
\usepackage[ibycus,english]{babel}

\newtheorem{definition}{Definition}[]
\newtheorem{lemma}{Lemma}[]

\newtheorem{remark}{Remark}[]
\newtheorem{corollary}{Corollary}[]

\newtheorem{proposition}{Proposition}[]

\newtheorem{assumption}{Assumption}[]

\newcommand{\R}{\ensuremath{\mathbb{R}}}

\begin{document}
	
	\title{Microscopic Derivation of Mean Field Game Models}
	
		\author{  Martin Frank\footnote{Karlsruhe Institute of Technology, Steinbuch Center for Computing, Hermann-von-Helmholtz-Platz 1, 76344 Eggenstein-Leopoldshafen, Germany}, Michael Herty\footnote{Institut f\"ur Geometrie und Praktische Mathematik, RWTH Aachen, Templergraben 55, 52056 Aachen, Germany}, Torsten Trimborn\footnotemark[2]~\footnote{Corresponding author: trimborn@igpm.rwth-aachen.de} }

	\maketitle


\abstract{Mean field game theory studies the behavior of a large number of interacting individuals in a game theoretic setting and has received a lot of attention in the past decade  \cite{lasry2007mean}.
In this work, we derive mean field game partial differential equation systems from deterministic microscopic agent dynamics. The dynamics are given by a particular class of ordinary differential equations, for which an optimal strategy can be computed \cite{bressan2011noncooperative}. We use the concept of Nash equilibria and apply the dynamic programming principle to derive the mean field limit equations and we study the scaling behavior of the system as the number of agents tends to infinity and find several mean field game limits. Especially we avoid in our derivation the notion of measure derivatives. Novel scales are motivated by an example of an agent-based financial market model. }	
	
   \section{Introduction}
 A great variety of phenomena in social and natural sciences is described and analyzed by agent-based models  \cite{gilbert2008agent, bellomo2017active, bellomo2019active}. 
 In many situations, the number of interacting agents is large and the agents compete against each other, e.g. by maximizing their individual payoff  \cite{axelrod1997complexity}.
 We assume that the agents' dynamics can be modeled by differential games \cite{nash1951non, aumann1964markets, gueant2011mean}. Models of this type can be found in many areas of research such as biology, engineering and economics. 
 Mean field game theory considers an infinite number of players that are in a Nash equilibria \cite{nash1951non, gueant2011mean}.
 In the last decade, there has been a vast number of contributions, since the independent introduction of mean field Nash equilibria by Lasry and Lions  and  Huang, Caines and Melham\'e  \cite{lasry2007mean, gueant2011mean, jovanovic1988anonymous, huang2006large, cardaliaguet2012long, cardaliaguet2013long, cardaliaguet2015master, lachapelle2016efficiency, lasry2008application, huang2006nash, huang2007invariance, bensoussan2013mean, achdou2010mean, carmona2013probabilistic, carmona2018probabilistic}. \\[2em]
 Mathematically speaking, a mean field game model is characterized by as system of  two coupled partial differential equations (PDEs) called Hamilton-Jacobi-Bellmann (HJB) equations.
The first PDE \eqref{first} is posed backwards in time and  associated to the optimal control problem.  The second PDE \eqref{second} is forward transport equation in time. \\
The prototype deterministic mean field game (MFG) model of the literature \cite{cardaliaguet2010notes, degond2014meanfield} is given by:
\begin{subequations}\label{prototype}
\begin{align}
&\partial_t h(t,x) - H(x, g(t,x), \nabla_x h(t,x) ) = 0,\label{first}\\
& \partial_t g(t,x)- div(\nabla_3 H(x,g(t,x), \nabla_x h(t,x) )\ g(t,x))=0, \label{second}\\
& g(0,x)=g_0,\quad h(x,T)=\mathcal{p}(x,g(T,x)),
\end{align}
\end{subequations}
with $t\in (0,T),\ x\in\R^d$ and Hamiltonian $H=H(x,g(t,x),  \nabla_x h(t,x))$. Here, $\nabla_3 H$ denotes the derivative with respect to the third component of the Hamiltonian. 
The function $h$ describes the value function of a player at time $t$ and attribute $x$, whereas
$g$ is a probability density function characterizing the probability of finding an agent at time $t$ with attribute $x$. 
Many results are analytical results such as existence and uniqueness \cite{bensoussan2014control, lasry2007mean, cardaliaguet2010notes, carmona2013probabilistic}.
The derivation of MFG systems from microscopic dynamics has been shown only by Lions \cite{cardaliaguet2010notes} to our knowledge and more recently by Degond et al. \citep{degond2014meanfield}. Recently, the solvability and mean field limit of linear-quadratic differential games has been investigated \cite{huang2018linear, ma2019linear, huang2019linear}.\\[2em]
In this work, we consider a particular class of microscopic differential games \cite{bressan2011noncooperative} and study different mean field limits,  obtained by different scaling assumptions.
Starting with disrete microsocpic dynamics, we solve the closed loop problem (feedback Nash equilibria) by the dynamic programming principle.
Then, we consider the mean field limit and the continuum limit. This approach enables us to avoid the measure derivatives as presented in \cite{cardaliaguet2010notes, degond2014meanfield}.\\
Another strategy in order to connect the microscopic dynamics with the mean field game model is presented in \cite{cardaliaguet2015master}. The authors first prove the well-posedness of the mean field game equations. As second step they evaluate the PDE along trajectories of the empirical measure in order to connect the limit equation with particle dynamics.\\ \\
We especially want to point out that several MFG models discussed in literature belong to our model \cite{cardaliaguet2010notes, gueant2011mean, degond2014meanfield, huang2018linear, ma2019linear, huang2019linear}. A particular example is the econophysical Levy-Levy-Solomon (LLS) model \cite{levy2000microscopic}, which does not lead to a MFG limit system of type \eqref{prototype}. 
This model is discussed in section 4. 

%
	
%
\section{Background and Main Results}
\paragraph{Microscopic Differential Game Model}
	We consider $N$ players, each player $i=1,...,N$ is  faced with the constrained optimization problem
	\begin{align}
		\argmax\limits_{u_i:\ [0,T]\to \mathbb{R}} {p}_i(\boldsymbol{x}(T))-\int\limits_0^T {L_i}(t,\boldsymbol{x},\boldsymbol{u})\ dt, \label{MicroGame}
	\end{align}
	with time horizon $T>0$ where $\boldsymbol{x}=(x_1,...,x_N)^{\top}\in \R^{d N},\ x_i=(x_i^1,...,x_i^d)^{\top}\in\R^d,\ 1\leq i\leq N,\ d\geq 1$ is the state solution to \eqref{stateDyn} and $ \boldsymbol{u}=(u_1,...,u_N)^{\top}\in \R^N$ is the control, $\boldsymbol{x}$ and $\boldsymbol{u}$ are time dependent. 
	\begin{align}\label{stateDyn}
		\dot{\boldsymbol{x}}=\boldsymbol{f}(t,\boldsymbol{x},\boldsymbol{u}),\quad \boldsymbol{x}(0)=\boldsymbol{x}_0\in\R^{dN}.
	\end{align}
The functions $\boldsymbol{f}=({f}_1,...,{f}_N)^{\top},\ \bs{L}=({L}_1,...,{L}_N)^{\top},\ \bs{p}=({p}_1,...,{p}_N)^{\top} $ are continuously differentiable and defined on:
\begin{align*}
{f}_i: [0,T]\times \R^{dN}\times \R^N\to \R^d,\quad {L}_i: [0,T]\times \R^{dN}\times \R^N\to \R,\quad {p}_i:\ \R^{dN}\to\R.
\end{align*}	
\paragraph{Nash Equilibria}
We assume that the players can observe the current state $\boldsymbol{x}$ of the system, thus their strategies may depend on time, the current state and  the initial state $u_i=u_i^*(t, \boldsymbol{x}(t);\boldsymbol{x}_0)$. Hence, the agents play a feedback or \emph{Markovian strategy}. This equilibrium concept is known as \emph{feedback Nash equilibrium}. For our further discussion we neglect the dependence of $\boldsymbol{u}$ on the trajectory $\boldsymbol{x}$ and the initial state $\boldsymbol{x}_0$. 
\begin{definition}
 A vector of control functions $u^*:  (t)\mapsto (u_1^*(t),...,u_N^*(t))^{\top}$ is a Nash equilibrium for the game \eqref{MicroGame} if the following holds.
 The control $u_i^*(t)$ provides a solution to the optimal control problem for player $i$, where we assume that the control of the other agents $\boldsymbol{u}^*_{-i}(t):= (u_1^*(t),...,u_{i-1}^*(t), u_{i+1}^*(t),... ,u_N^*(t))^{\top}\in\R^{N-1}$ are fixed and optimal:
\begin{subequations}\label{NashSys}
\begin{align}
&\argmax\limits_{u_i: [0,T]\to\R} p_i(\boldsymbol{x}(T))-\int\limits_0^{\top} L_i(t,\boldsymbol{x},u_i, \boldsymbol{u}^*_{-i})\ dt,\\
& \text{s.t.}\quad  \dot{\boldsymbol{x}}=\boldsymbol{f}(t,\boldsymbol{x},\boldsymbol{u}),\quad \boldsymbol{x}(0)=\boldsymbol{x}_0.
\end{align}
\end{subequations}
\end{definition}

\begin{assumption} \label{assI}
We assume that the dynamics and the running costs decouple in the following way:
\begin{subequations}\label{MicroMod}
\begin{align}
\bs{f}(t,\boldsymbol{x},\boldsymbol{u})&=\bs{f^0}(t,\boldsymbol{x})+\bs{M}^1(t,\boldsymbol{x})\ u_1(t,\bs{x})+...+\bs{M}^N(t,\boldsymbol{x})\ u_N(t,\bs{x}),\\
{ L}_i(t,\boldsymbol{x}, \boldsymbol{u})&= L_i^1(t,\boldsymbol{x},u_1)+...+L_i^N(t,\boldsymbol{x},u_N),
\end{align}
\end{subequations}
with $\boldsymbol{u}(t,\bs{x})=(u_1(t,\bs{x}),...,u_N(t,\bs{x}))^{\top}\in \R^N, \bs{f^0}(t,\boldsymbol{x})\in\R^{dN},\ \bs{M}^i(t,\bs{x})\in\R^{dN},\ i=1,...,N$. 
\end{assumption}
For this class of differential games, one can prove the existence and uniqueness of Nash equilibrium solutions \citep{bressan2011noncooperative}. 
The following proposition gives detailed conditions on the existence and uniqueness of Nash equilibrium solutions, see Bressan \citep{bressan2011noncooperative}. 

\begin{proposition}\label{theo1}
[Lemma 2 in \citep{bressan2011noncooperative}] We consider the microscopic model \eqref{MicroMod} and assume that $\bs{M}^i$ depends continuously on $(t,\boldsymbol{x})$. For every $(t,\boldsymbol{x})\in[0,T]\times \R^{dN}$ and any vector $\lambda^i\in\R^{dN},\ i=1,...,N$, there exist a unique vector $\boldsymbol{u}^*\in\R^N$, respectively a map:
$(t,\boldsymbol{x}, \boldsymbol{\lambda}) \mapsto \boldsymbol{u}^*(t,\boldsymbol{x}, \boldsymbol{\lambda}), $ such that the control value $u_i^*$ fulfills
\[
u_i^*=\argmax\limits_{\omega\in\R}\quad  (\lambda^i)^{\top} \bs{M}^i(t,\boldsymbol{x})\ \omega-L_i^i(t,\boldsymbol{x}, \omega).
\]
\end{proposition}

Here, $\lambda^i$ denotes the Lagrange multiplier of the optimality system \eqref{NashSys}. 
In order to compute the control explicitly, we restrict ourselves to the following class of 
running costs $\bs{L}$.
\begin{assumption} \label{cost}
We consider running costs of the form
\begin{subequations} \label{runCost}
\begin{align}
&{L}_i^i(t,\boldsymbol{x}, \boldsymbol{u})=\frac1Nl_{i}(t,\boldsymbol{x})+ \frac{\widehat{\alpha}}{2}\ u_i^2,\\
& {L}_i^k(t,\boldsymbol{x}, \boldsymbol{u})=\frac1N l_{i}(t,\boldsymbol{x})+ \frac{\bar{\alpha}}{2}  u_k^2,\ k\neq i. 
\end{align}
\end{subequations}

with $\widehat{\alpha}>0,\ \bar{\alpha}\geq 0$ and $l_i: [0,T]\times \R^{dN}\to \R,$ continuous differentiable in $(t,\boldsymbol{x})$. 
\end{assumption}

\begin{lemma}
Under Assumption \ref{cost} the optimal control value $u_i^*$ is given by
\[
u_i^*=\frac{1}{\widehat{\alpha}} (\lambda^i)^{\top}\ \bs{M}^i(t,\boldsymbol{x}).
\]
\end{lemma}

In the further discussion we summarize the main steps and key assumptions needed to derive the MFG limit system.

\paragraph{Symmetry Assumptions}
A crucial assumption is to consider identical, indistinguishable players. This translates into symmetry assumptions on our running costs $\bs{L}$ and dynamics $\bs{f}$.  
\begin{definition}\label{EmpDef}
We define an empirical moment $\rho^N_{\Phi}(\boldsymbol{x})\in\R,\ \bs{x}\in\R^{dN}$ of the polynomial of degree $n$
$$
\Phi(x) = \sum\limits_{|j|\leq n}  \beta_{j } \ (x^1)^{j_1}  (x^2)^{j_2}...  (x^d)^{j_d},\ \beta_j\in\R\,\ x=(x^1,...,x^d)^{\top}\in\R^d,
$$
 with multi-index notation $|j|=j_1+j_2+...+j_d$,   
 by
$$
\rho^N_{\Phi}(\bs{x}):= \frac1N \sum\limits_{k=1}^N \Phi(x_k).
$$
\end{definition}
Notice that the empirical moment $\rho^N_{\Phi}$ is symmetric with respect to the $N$ variables $x_k$. Here, symmetry $\rho^N_{\Phi}(\boldsymbol{x})=\rho^N_{\Phi}(\boldsymbol{x}_{\sigma}) $ is defined by any permutation $\sigma: \{1,...,N\}\to \{1,...,N \},\ \boldsymbol{x}_{\sigma}:=(x_{\sigma_1},...,x_{\sigma_N})$.
\begin{assumption}\label{ASym}
 We model the dynamic $f$ and running cost $L$ for some $t\in[0,T],\ \bs{x}\in\R^{dN},\ \rho^N_{\Phi}(\bs{x})\in\R$ as follows.
\begin{itemize}
\item[i)] $\bs{f_i^0}(t,\boldsymbol{x}):=f(t,x_i,\rho^N_{\Phi}(\bs{x}))$ with $f: \ [0,T]\times \R^d \times \R\to \R^d$, $i=1,...,N$.
\item[ii)] $\boldsymbol{M}^i_i(t,\boldsymbol{x}):=\widehat{m}(t,x_i,\rho^N_{\Phi}(\bs{x})),\ \boldsymbol{M}_k^i(t,\boldsymbol{x}):=\bar{m}(t,x_k,\rho^N_{\Phi}(\bs{x})),\ k\neq i$,  $k=1,...,N$ with $\widehat{m}:  [0,T]\times \R^d\times \R\to \R^d,\ \bar{m}: [0,T]\times \R^d\times \R\to \R^d$.
\item[iii)]  $l_i(t,\boldsymbol{x}):= l(t,x_i, \rho_{\Phi}^N(\bs{x}))$ with $l: [0,T]\times \R^d\times  \R\to \R$.
 \item[iv)] ${p}_i(\boldsymbol{x}):= p(x_i,  \rho_{\Phi}^N(\bs{x}))$  with $p: \R^d \times \R\to \R$. 
\end{itemize}
Furthermore, we assume that the functions $f,\widehat{m}, \bar{m},l,p$ are 
\begin{itemize}
\item continuously differentiable in all arguments,
\item Lipschitz continuous with respect to $(\boldsymbol{x}, \rho_{\Phi}^N)$  and with Lipschitz constant independent of $t$.
\end{itemize}
\end{assumption}
Notice that all quantities denoted with a hat denote the self-interaction of agents, whereas the quantities with a bar denote the interaction with other agents. Furthermore, we want to emphasize that the functions $f,\ \widehat{m},\ \bar{m},\ l,\ p $ are symmetric in $\boldsymbol{x}_{-i}:=(x_1,...,x_{i-1},x_{i+1},...,x_N)^{\top}\in\R^{d (N-1)}$.

\paragraph{Scaling Assumptions}
The mean field limit will be derived for different scalings of $ \widehat{m}, \bar{m},  \widehat{\alpha}, \bar{\alpha} $. 
In order to obtain in the limit $N\to \infty$ a mean field game equation the quantities $\bar{m}$ and $\bar{\alpha}$ need to be scaled in the following way:
\begin{assumption}\label{ScaleA}
For $\bar{a},\bar{\theta}\geq 1$ we define
\begin{align*}
\boldsymbol{\bar{m}}:=\frac{1}{N^{\bar{\theta}}} \bar{m},\quad \boldsymbol{\bar{\alpha}}:=\frac{1}{N^{\bar{a}}} \bar{\alpha}. 
\end{align*}
For $\widehat{\theta}, \widehat{a} \geq 0$ we define
\begin{align*}
&\boldsymbol{\widehat{m}} := \frac{1}{N^{\widehat{\theta}}} \widehat{m } ,\quad \boldsymbol{\widehat{\alpha}}:=\frac{1}{N^{\widehat{a}}} \widehat{\alpha}.
\end{align*}
\end{assumption}
The previous assumption will be explained in detail in section \ref{Main}. 
The main result is as follows, see Lemma \ref{ScaleTheo}. If the inequalities
\begin{align*}
&i)\ \quad \text{\ibygr{h}}_1 := \widehat{a}-2\ \widehat{\theta} \leq  0,\\
&ii)\ \quad \text{\ibygr{h}}_2:= \widehat{a}+1-\widehat{\theta}-\bar{\theta} \leq 0,\\
&iii)\ \quad \text{\ibygr{h}}_3:=2\ \widehat{a}+1-2\ \widehat{\theta}-\bar{a} \leq 0,\\
\end{align*}
 hold under Assumptions \ref{assI}-\ref{ScaleA} a formal mean field limit of the differential game \eqref{MicroGame}-\eqref{stateDyn} exists.
Eight different limiting equations are obtained depending if the inequalities $i)-iii)$ are sharp or not. In the case that the inequalities $i)-iii)$ are sharp we call this \textbf{vanishing-coupling} since all game-theoretic components vanish in the mean field limit. The case of equality in $i)$ and inequality in $ii)$ and $iii)$ corresponds to the case in literature e.g.  \cite{degond2014meanfield, cardaliaguet2010notes}. For that reason we call this setting the \textbf{classical-coupling}. In the case that we have equality in $ii)$ or $iii)$ we obtain non-local integral terms in our limit equations. Therefore, we refer to this setting as \textbf{non-local-coupling}. 

\paragraph{Main Result}
The formal MFG limit under different scalings is given by \eqref{MFGFirst}:
\begin{subequations}\label{MFGFirst}
\begin{align} 
&\partial_t h(t,x) +  {f}(t,x,\rho_{\Phi}[g])  \ \nabla_x h(t,x)  \\
&\quad\quad +\left( \chi_{\{\text{\ibygr{h}}_2=0\}}\left[\frac{1}{\widehat{\alpha}}\ \bar{{m}}(t,x,\rho_{\Phi}[g]) \int g(t,z)\ \widehat{{m}}(t,z,\rho_{\Phi}[g])\ \nabla_z h(t,z)\ dz \right] \right) \ \nabla_x h(t,x)\\
&\quad\quad={l}(t,x,\rho_{\Phi}[g])-\chi_{\{\text{\ibygr{h}}_1=0\}}\left[\frac{1}{2 \widehat{\alpha}} (\widehat{{m}}(t,x,\rho_{\Phi}[g])\  \nabla_x h(t,x))^2\right]\\
&\quad\quad\quad\quad\quad + \chi_{\{\text{\ibygr{h}}_3=0\}}\left[ \frac{\bar{\alpha}}{2\ \widehat{\alpha}^2} \int g(t,z)\ (\nabla_{z} h(t,z)\ \widehat{{m}}(t,z, \rho_{\Phi}[g]) )^2\ dz\right],  \\
&\partial_t g(t,x) + div_x \Bigg( \Big( {f}(t,x,\rho_{\Phi}[g]) + \chi_{\{ \text{\ibygr{h}}_1=0\}}  \left[\frac{1}{\widehat{\alpha}}   (\widehat{{m}}(t,x,\rho_{\Phi}[g]))^2 \ \nabla_x h(t,x)\right]   \\
& \quad\quad+ \chi_{\{\text{\ibygr{h}}_2=0\}}\left[\frac{1}{\widehat{\alpha}} \bar{{m}}(t,x,\rho_{\Phi}[g]))\ \int g(t,z)\ \nabla_z h(t,z)\  \widehat{{m}}(t,z,\rho_{\Phi}[g])\ dz \right]\ \Big)\ g(t,x)\Bigg)=0,\\
& g(0,x)=g_0(x),\quad h(T,x)={p}(x,\rho_{\Phi}[g](T)).
\end{align}
\end{subequations}
Here,  $\chi$ denotes the indicator function and $$\rho_{\Phi}[g]=\rho_{\Phi}[g](t) = \int\limits_{\R^d} \Phi(y)\ g(t,y)\ dy, $$
the moment. 
In the case of the classical coupling $\chi_{\{\text{\ibygr{h}}_2=0\}}= \chi_{\{\text{\ibygr{h}}_3=0\}}=0$ holds. For the vanishing-coupling $\chi_{\{\text{\ibygr{h}}_1=0\}}=\chi_{\{\text{\ibygr{h}}_2=0\}}= \chi_{\{\text{\ibygr{h}}_3=0\}}=0$ holds. The terms multiplied by $\chi_{\{\text{\ibygr{h}}_1=0\}}$ represent the self-interaction of each player with the field of other players. Whereas the terms multiplied by $\chi_{\{\text{\ibygr{h}}_2=0\}}$ are the interactions of each player with the field of players. Finally, the quantity multiplied by $\chi_{\{\text{\ibygr{h}}_3=0\}}$ appears when the control costs of each agent depends on the field of optimal controls of the other agents.

\section{ Main Result} \label{Main}

We use dynamic programming principle to solve the closed loop problem \cite{bellman2013dynamic, bertsekas1995dynamic}. The value function $V_i(\tau, \boldsymbol{y}),\ \tau\in[0,T],\ \boldsymbol{y}\in\R^{dN}$ is defined as:
\[
V_i(\tau, \boldsymbol{y}):=p_i(\boldsymbol{x}^*(T))- \int\limits_{\tau}^{\top} L_i(t,\boldsymbol{x}^*(t), \boldsymbol{u}^*(t),\boldsymbol{x}^*(t))\ dt,
\]
where $x^*_i(t)=x_i^*(t,\tau,\boldsymbol{y})$ denotes the optimal solution and the vector $\boldsymbol{y}\in\R^{dN}$ is the initial condition of our equation \eqref{stateDyn} at time $\tau$: $\boldsymbol{x}(\tau)=\boldsymbol{y}$. Consequently, the value function $V_i$ is the total payoff for the i-th player at time $\tau$ with initial condition $\boldsymbol{y}$. 

\begin{lemma}
For a small time step $\Delta \tau $ the discrete dynamic programming principle applied to the differential game \eqref{MicroGame}-\eqref{stateDyn} for $i=1,...,N$ is given by:
\begin{subequations}\label{ValueSys}
\begin{align}
V_i(\tau, \boldsymbol{y}) = & 	 - \Delta \tau \Bigg( l(\tau,x_i(\tau), \rho^N_{\Phi} (\boldsymbol{x}(\tau)))+ \frac{\widehat{\alpha}}{2}\ (u_i^*(\tau))^2 +\frac{\bar{\alpha}}{2} \sum\limits_{k=1,\ k\neq i}^N\ (u_k^*(\tau))^2\Bigg)\\
&+ V_i(\tau+ \Delta \tau, \boldsymbol{x}(\tau+\Delta \tau))+ \mathcal{O}(\Delta \tau^2) ,\\
u_i^*(\tau)= &\frac{1}{\widehat{\alpha}} \widehat{m}(\tau, x_i(\tau), \rho^N_{\Phi}(\boldsymbol{x}(\tau)))\ \nabla_{x_i} V_i(\tau+\Delta \tau, \boldsymbol{x}(\tau+\Delta \tau))\\
& +\frac{1}{\widehat{\alpha}}  \sum\limits_{k=1,\ k\neq i}^N \bar{m}(\tau, x_k(\tau), \rho^N_{\Phi}(\boldsymbol{x}(\tau))) \nabla_{x_k} V_i(\tau+\Delta \tau, \boldsymbol{x}(\tau+\Delta \tau))+\mathcal{O}(\Delta \tau^2).
\end{align}
\end{subequations}
\begin{small}
\begin{align}
\begin{cases}
&{x}_i(\tau+\Delta \tau)= y_i+ \Delta \tau\ \Bigg[ f(\tau,x_i(\tau), \rho^N_{\Phi}(\boldsymbol{x}(\tau)))+ \widehat{m}(\tau,x_i(\tau), \rho^N_{\Phi}(\boldsymbol{x}(\tau)))\  u_i^*(\tau)  \\
&+ \bar{m}(\tau,x_i(\tau), \rho^N_{\Phi}(\tau))\ \sum\limits_{k=1,  k\neq i}^N u_k^*(\tau)\Bigg] + \mathcal{O}(\Delta t^2),\\
&x_i(\tau)=y_i. 
\end{cases}
\end{align}
\end{small}
\end{lemma}

\begin{proof}
For any time step $\Delta \tau>0$ by the dynamic programming principle 
\begin{align}
V_i(\tau, \boldsymbol{y}) =  	\max\limits_{u_i: [\tau,\tau + \Delta \tau]\to \R} -\int\limits_{\tau}^{\tau+\Delta \tau } L_i(t,\boldsymbol{x},u_i, \boldsymbol{u}^*_{-i})\ dt  + V_i(\tau+ \Delta \tau, \boldsymbol{x}(\tau+\Delta \tau)),\label{BellmanPrinciple}
\end{align}
holds. As $\Delta \tau$ shrinks to zero the choice of the unknown reduces to the choice of $u^*_i(\tau)$
\begin{subequations}
\begin{align}
V_i(\tau, \boldsymbol{y}) =  	\max\limits_{u_i(\tau) \in \R}& - \Delta \tau \left( l(\tau,x_i(\tau), \rho^N_{\Phi} (\boldsymbol{x}(\tau)))+ \frac{\widehat{\alpha}}{2}\ (u_i(\tau))^2+\frac{\bar{\alpha}}{2} \sum\limits_{k=1,\ k\neq i}^N\ (u_k^*(\tau))^2\right) \\
&+ V_i(\tau+ \Delta \tau, \boldsymbol{x}(\tau+\Delta \tau))+ \mathcal{O}(\Delta \tau^2).\label{ApproxBell}
\end{align}
\end{subequations}
The first order approximation of our state dynamics on $\tau \leq t \leq \tau+\Delta \tau$ are given by:
\begin{subequations}
\begin{align}
x_i(\tau+\Delta \tau)=y_i+ \Delta \tau \Bigg( & f(\tau,x_i(\tau), \rho^N_{\Phi}(\boldsymbol{x}(\tau)))+ \widehat{m}(\tau,x_i(\tau), \rho^N_{\Phi}(\boldsymbol{x}(\tau))) \ u_i(\tau) \\
&+ \bar{m}(\tau,x_i(\tau), \rho^N_{\Phi}(\boldsymbol{x}(\tau)))\ \sum\limits_{k=1,  k\neq i}^N u_k^*(\tau)    \Bigg) + \mathcal{O}(\Delta t^2).\label{Euler}
\end{align}
\end{subequations}
We solve the optimization problem on $(\tau, \tau+ \Delta t)$ and obtain by the necessary condition the value $u_i^*(\tau)$
\begin{subequations}
\begin{align}
u_i^*:=u_i^*(\tau)= &\frac{1}{\widehat{\alpha}} \widehat{m}(\tau, x_i(\tau), \rho^N_{\Phi}(\boldsymbol{x}(\tau)))\ \nabla_{x_i} V_i(\tau+\Delta \tau, \boldsymbol{x}(\tau+\Delta \tau))\\
& +\frac{1}{\widehat{\alpha}}  \sum\limits_{k=1,\ k\neq i}^N \bar{m}(\tau, x_k(\tau), \rho^N_{\Phi}(\boldsymbol{x}(\tau))) \nabla_{x_k} V_i(\tau+\Delta \tau, \boldsymbol{x}(\tau+\Delta \tau))+\mathcal{O}(\Delta \tau^2).
\end{align}
\end{subequations}

\end{proof}

\paragraph{Averaged Bellmann Principle}
In this section, we reduce the large system \eqref{ValueSys} of $N$ value functions $V_i$ to an averaged value function $\langle W \rangle_N$. 
The goal is to show the value function $V_i$ only depends on $y_i$ and on all the other variables $\boldsymbol{y}_{-i}:= (y_1,...,y_{i-1}, y_{i+1},...,y_N)^{\top}\in\R^{(N-1)\ d}$ only through the empirical moment $\rho^{N-1}_{\Phi,i}:=\rho^{N-1}_{\Phi}(\boldsymbol{x}_{-i})= \frac{1}{N-1} \sum\limits_{k=1,\\ k\neq i}^N \Phi(\boldsymbol{x}_k)$. The value function $W$ then only depends on $y_i$ and the empirical moment $\rho^{N-1}_{\Phi, i}$. Without any rescaling it is in general not possible to identify $V_i$ by a value function $W$ since the system \eqref{ValueSys} as well as the state dynamics \eqref{Euler} depend on $x_j, j\neq i$ not only by empirical moments. In fact, we  introduce a mean field scaling such that a dependence of the value function $W$ is only on $\rho^N_{\Phi,i}$. 
\begin{definition}\label{GlobAs}
Assume that Assumption \ref{ASym} is satisfied for a fixed but arbitrary function $\Phi(\cdot)$. Then, we define the averaged  Bellman system depending on two scaling parameter $\bar{a},\bar{\theta}\geq1$. 
We call a $C^2$ function  $W: [\tau, \tau +\Delta \tau]\times \R^d \times \R\to \R$ a solution to the HJB system if for all $y_i\in\R^d$ and  $\rho^{N-1}_{\Phi, i}\in\R$ equations \eqref{SymBellman} hold. 

\begin{subequations}\label{SymBellman}
\begin{align} 
&\left\langle W(\tau, y_i, \rho^{N-1}_{\Phi, i})\right\rangle_N =  	 - \Delta \tau \Bigg\langle \Bigg( l(\tau,y_i, \rho^N_{\Phi} )+ \frac{\widehat{\alpha}}{2}\ (p_i^*(\tau))^2 +\frac{\bar{\alpha}}{N^{\bar{a}}\ 2}\sum\limits_{k=1,\ k\neq i}^N\ (p_k^*(\tau))^2 \Bigg)\label{scaleA}\\
&\quad+ W\Big(\tau+ \Delta \tau,v_i(\tau+\Delta \tau), \rho^{N-1}_{\Phi, i}(\tau+\Delta \tau)\Big)\Bigg\rangle_N,\\
&\begin{cases}
 v_i(\tau+\Delta \tau)  =  y_i+\Delta \tau \Bigg[ f(\tau,y_i, \rho^N_{\Phi})+ \widehat{m}(\tau,y_i, \rho^N_{\Phi})\ p_i^*(\tau) \\
\quad +     \frac{1}{N^{\bar{\theta}}} \bar{m}(\tau,y_i, \rho^N_{\Phi}) \sum\limits_{k=1,  k\neq i}^N p_k^*(\tau)\Bigg],\label{scaleM}\\
 v_i(\tau)=  y_i,\quad i=1,...,N,
\end{cases}\\
&p^*_i:=\frac{1}{\widehat{\alpha}} \widehat{m}(\tau, v_i(\tau), \rho^N_{\Phi})\ \nabla_{v_i} {W}(\tau+\Delta \tau, v_i(\tau+\Delta t),  \rho^{N-1}_{\Phi, i}(\tau+\Delta \tau)).
\end{align}
\end{subequations}
\end{definition}
Here, we use $\rho^{N-1}_{\Phi, i}(\tau+\Delta \tau)$ as short hand notation for $\rho^{N-1}_{\Phi}(\boldsymbol{v}_{-i}(\tau+\Delta \tau))$
and  $\langle \cdot\rangle_{N}$ denote the averaging with respect to all agents: $\langle \cdot\rangle_{N}:= \frac1N \sum\limits_{i=1}^N (\cdot)$.
The function $W(\tau, y_i,\rho^{N-1}_{\Phi, i})$ is not symmetric in $y_i$ but the empirical moment $\rho^{N-1}_{\Phi, i}$ is symmetric with respect to the variables $\boldsymbol{y}_{-i}:=(y_1,...,y_{i-1}, y_{i+1},...,y_N)^{\top}\in \R^{d(N-1)}$. As next step we connect $W$ with the value function of the $i-$th agent. 
\begin{corollary}\label{Connect}
Assume that a unique $C^2$ solution $\langle W\rangle_N$ of \eqref{SymBellman} with respect to the second and third component exists. 
Additionally we assume that $\bar{W}: [\tau, \tau+\Delta \tau]\times \R^d \times \R\to \R$ is a unique $C^2$ solution in the sense of Definition \ref{GlobAs} of the system \eqref{eqCon}-\eqref{optu}
\begin{subequations}
\begin{align} \label{eqCon}
\langle \bar{W}(\tau,y_i,\bar{\rho}^{N-1}_{\Phi,i})\rangle_N = &  - \Delta \tau  \left\langle \Bigg( l(\tau,y_i, \bar{\rho}^N_{\Phi,i} )+ \frac{\widehat{\alpha}}{2}\ (\bar{u}_i^*(\tau))^2 +\frac{1}{N^{\bar{a}}}\frac{\bar{\alpha}}{2} \sum\limits_{k=1,\ k\neq i}^N\ (\bar{u}_k^*(\tau))^2\Bigg) \right\rangle_N \\
&+ \langle \bar{W}(\tau+ \Delta \tau ,\bar{x}_i(\tau + \Delta \tau),\bar{\rho}^{N-1}_{\Phi,i}(\tau+\Delta \tau)) \rangle_N,
\end{align}
\end{subequations}
\begin{small}
\begin{align}\label{eqConDyn}
\begin{cases}
&\bar{x}_i(\tau+\Delta \tau)= y_i+ \Delta \tau\ \Bigg[ f(\tau,y_i, \bar{\rho}^N_{\Phi,i})+ \widehat{m}(\tau,y_i,\bar{ \rho}^N_{\Phi,i})\  \bar{u}_i^*(\tau)  \\
&+ \frac{1}{N^{\bar{\theta}}}\bar{m}(\tau,y_i, \bar{\rho}^N_{\Phi,i})\ \sum\limits_{k=1,  k\neq i}^N \bar{u}_k^*(\tau)\Bigg],\\
&\bar{x}_i(\tau)=y_i. 
\end{cases}
\end{align}
\end{small}
with
\begin{subequations}\label{optu}
\begin{align}
\bar{u}_i^*(\tau):=  &\frac{1}{\widehat{\alpha}} \widehat{m}(\tau, y_i, \bar{\rho}^N_{\Phi,i})\ \nabla_{\bar{x}_i} \bar{W}(\tau+ \Delta \tau ,\bar{x}_i(\tau + \Delta \tau),\bar{\rho}^{N-1}_{\Phi,i}(\tau + \Delta \tau))\\
& +\frac{1}{\widehat{\alpha}} \frac{1}{N^{\bar{\theta}}} \sum\limits_{k=1,\ k\neq i}^N \bar{m}(\tau, y_k, \bar{\rho}^N_{\Phi,i})\  \partial_3 \bar{W}(\tau+ \Delta \tau ,\bar{x}_i(\tau + \Delta \tau),\bar{\rho}^{N-1}_{\Phi,i}(\tau+\Delta \tau))\frac{ \nabla_{x_k} \Phi(x_k)}{N}   . 
\end{align}
\end{subequations}
Furthermore, we assume
\begin{align*}
& |\bar{W}(\tau+\Delta \tau,z,\rho)-W(\tau+\Delta \tau,\tilde{z},\tilde{\rho} )|\leq \frac{C}{N^{\bar{\theta}}}(|z-\tilde{z}|+ |\rho-\tilde{\rho}|),\ \forall \rho, \tilde{\rho}\in\R,\  z,\tilde{z} \in \R^d,\\
&\Big| \nabla_{z} \bar{W}(\tau+ \Delta \tau ,z,\rho)-\nabla_{z} W(\tau+ \Delta \tau ,\tilde{z},\tilde{\rho})  \Big| \leq \frac{C}{N^{\bar{\theta}}} (|z-\tilde{z}|+ |\rho-\tilde{\rho}|),\  \forall \rho, \tilde{\rho}\in\R,\  z, \tilde{z} \in \R^d,    \\
& | \Phi(z) - \Phi(\tilde{z})| \leq C |z-\tilde{z}|,\ \forall z,\tilde{z} \in \R^d,\\
&|\widehat{m}(t, z, \rho)| \leq C,\quad |\bar{m}(t, z, \rho)| \leq C,\quad |\nabla \Phi( z)| \leq C,\quad\forall\  t\in[\tau, \tau+\Delta \tau],\  \rho\in\R,\  z \in \R^d.
\end{align*}
Then the following inequalities 
\begin{align*}
&|\langle \bar{W}(\tau, y_i, \rho)-{W}(\tau, y_i, \rho)\rangle_N|\leq \frac{C}{N^{\bar{\theta}}} ,\quad \forall\ \rho\in\R,\\
&|\langle \nabla_{y_i}\bar{W}(\tau, y_i,\rho )-  \nabla_{y_i}{W}(\tau, y_i, \rho) \rangle_N|\leq \frac{C}{N^{\bar{\theta}}}  , \quad \forall\ \rho\in\R,\\
& | \langle \bar{x}_i(\tau+\Delta \tau)-v_i(\tau+\Delta \tau)\rangle_N| \leq \frac{C}{N^{\bar{\theta}}},
\end{align*}
 hold.
\end{corollary}
\begin{proof}
For simplicity we only consider the $i$-th value function and do not consider the average,  $\langle \cdot \rangle_N $ . First, we note that 
\begin{align*}
 \partial_3 \bar{W}(\tau+\Delta \tau ,\bar{x}_i,\rho^{N-1}_{\Phi}(\boldsymbol{\bar{x}}_{-i}))\ \frac{\nabla_{\bar{x}_k} \Phi(\bar{x}_k)}{N}
 = \nabla_{\bar{x}_k} \bar{W}(\tau+\Delta \tau , \bar{x}_i,\rho^{N-1}_{\Phi}(\boldsymbol{\bar{x}}_{-i}))    ,
\end{align*}
for $k\neq i$ holds. Thus, we conclude that
$$
|\nabla_{\bar{x}_k} \bar{W}(\tau+\Delta \tau , \bar{x}_i(\tau+\Delta \tau),\rho^{N-1}_{\Phi}(\boldsymbol{\bar{x}}_{-i}(\tau+\Delta \tau))) | \leq \frac{C}{N},
$$
holds since $|\nabla \Phi(\bar{x}_k)| <C$ and  $|\partial_3 \bar{W}| <C$ by assumption. Consequently 
$$
\Big|\frac{1}{\widehat{\alpha}} \frac{1}{N^{\bar{\theta}}} \sum\limits_{k=1,\ k\neq i}^N \bar{m}(\tau, y_k, \bar{\rho}^N_{\Phi,i}) \nabla_{\bar{x}_k} \bar{W}(\tau+ \Delta \tau ,\bar{x}_i(\tau + \Delta \tau),\rho^{N-1}_{\Phi}(\boldsymbol{\bar{x}}_{-i}(\tau+\Delta \tau))) \Big| \leq \frac{C}{N^{\bar{\theta}}},
$$
holds as well. 
Therefore, we obtain:
\begin{align*}
\Big| \bar{u}_i^*(\tau) - \frac{1}{\widehat{\alpha}} \widehat{m}(\tau, y_i, \bar{\rho}^N_{\Phi,i})\ \nabla_{\bar{x}_i} \bar{W}(\tau+ \Delta \tau , \bar{x}_i(\tau + \Delta \tau),\bar{\rho}^{N-1}_{\Phi,i} (\tau+\Delta \tau))\Big|\leq \frac{C}{N^{\bar{\theta}}}.
\end{align*}
Hence, the following inequality
\begin{align*}
|\bar{u}_i^*(\tau)-p_i^*(\tau)|&\leq \frac{C}{N^{\bar{\theta}}}+ C\ \Big| \nabla_{\bar{x}_i} \bar{W}(\tau+ \Delta \tau ,\bar{x}_i(\tau + \Delta \tau),\bar{\rho}^{N-1}_{\Phi, i}(\tau+\Delta \tau))\\
&\quad\quad\quad\quad-\nabla_{v_i} W(\tau+ \Delta \tau ,v_i(\tau + \Delta \tau),\rho^{N-1}_{\Phi, i}(\tau +\Delta t))  \Big|\\
&\stackrel{Ass}{\leq} \frac{C}{N^{\bar{\theta}}} + \frac{C}{N^{\bar{\theta}}}\ |\bar{x}_i(\tau+\Delta \tau)- v_i(\tau+\Delta \tau)  | \\
&\quad +  \frac{C}{N^{\bar{\theta}}}\ |\Phi\big( \bar{x}_k(\tau+\Delta \tau)\big)- \Phi\big( v_k(\tau+\Delta \tau) \big) |\\
&\leq  \frac{C}{N^{\bar{\theta}}} + \frac{C}{N^{\bar{\theta}}}\ |\bar{x}_i(\tau+\Delta \tau)- v_i(\tau+\Delta \tau)  |,
\end{align*}
is fulfilled. Then it immediately follows that
\begin{align*}
|\bar{x}_i(\tau+\Delta t) - v_i(\tau+\Delta t) | &\leq    | f(\tau, y, \bar{\rho}^{N-1}_{\Phi,i})- f(\tau, y, \rho^{N-1}_{\Phi, i}) |+   C\  |u_i^*(\tau)-p_i^*(\tau)| \\
&\quad + \frac{C}{N^{\bar{\theta}}}  N \ |u_k^*(\tau)-p_k^*(\tau)|    \\
&\stackrel{Sym}{\leq} 0+ \Big(  C+ \frac{C}{N^{\bar{\theta}}} N \Big) \Big( \frac{C}{N^{\bar{\theta}}} + \frac{C}{N^{\bar{\theta}}}\ |\bar{x}_i(\tau+\Delta t) - v_i(\tau+\Delta t) |     \Big)\\
& \leq  \frac{C}{N^{\bar{\theta}}} + \frac{C}{N^{\bar{\theta}}}\ |\bar{x}_i(\tau+\Delta t) - v_i(\tau+\Delta t) | ,
\end{align*}
holds and we get:
\begin{align*}
|\bar{x}_i(\tau+\Delta t) - v_i(\tau+\Delta t) | \leq \frac{C}{N^{\bar{\theta}}} \Big(  1-\frac{C}{N^{\bar{\theta}}} \Big)^{-1} = \frac{C}{N^{\bar{\theta}}-C}\leq \frac{C}{N^{\bar{\theta}}},
\end{align*}
for $N$ sufficiently large. 
As next step we use the assumption  
\begin{align*}
& |\bar{W}(\tau+\Delta \tau,\bar{z},\bar{\rho})-W(\tau+\Delta \tau,z,\rho)|\leq \frac{C}{N^{\bar{\theta}}} (|\bar{z}-z|+ |\bar{\rho}-\rho|),\\
\end{align*}
to obtain the following estimate:
\begin{align*}
& |\bar{W}(\tau,y_i,\bar{\rho}^{N-1}_{\Phi, i})-W(\tau,y_i,\rho^{N-1}_{\Phi, i})| \\
&\leq    |\bar{W}(\tau+\Delta \tau ,\bar{x}_i(\tau+\Delta \tau),\bar{\rho}^{N-1}_{\Phi,i}(\tau + \Delta \tau))-W(\tau+\Delta \tau ,v_i(\tau+\Delta \tau),\rho^{N-1}_{\Phi, i}(\tau +\Delta \tau))|    \\
&+\Bigg|\frac{\widehat{\alpha}}{2}\ [(\bar{u}_i^*(\tau))^2-(p_i^*(\tau))^2]+\frac{\bar{\alpha}}{2} \frac{1}{N^{\bar{a}}}\sum\limits_{k=1,\ k\neq i}^N\ (\bar{u}_k^*(\tau))^2-(p_k^*(\tau))^2 \Bigg| \\
&\leq \frac{C}{N^{\bar{\theta}}} \ \Big|(|\bar{x}_i(\tau+\Delta \tau)-v_i(\tau+\Delta \tau)  | + |\bar{\rho}^{N-1}_{\Phi,i}(\tau + \Delta \tau)-\rho^{N-1}_{\Phi, i}(\tau +\Delta \tau) |\Big)       \\
&\quad \quad+ \Big(   \frac{\widehat{\alpha}}{2} +\frac{\bar{\alpha}}{2}  \frac{N}{N^{\bar{a}}}  \Big)   \Big| (\bar{u}_i^*(\tau))^2-(p_i^*(\tau))^2 \Big|   \\
&\leq   \frac{C}{N^{\bar{\theta}}}\Big( \frac{C}{N^{\bar{\theta}}} + \frac{C}{N^{\bar{\theta}}}   \Big)+ \Big(   \frac{\widehat{\alpha}}{2} +\frac{\bar{\alpha}}{2}  \frac{N}{N^{\bar{a}}}  \Big) \Big|(\frac{1}{\widehat{\alpha}} \widehat{m}(\tau, y_i, \bar{\rho}^N_{\Phi,i})\ \nabla_{\bar{x}_i} \bar{W}(\tau+ \Delta \tau ,\bar{x}_i(\tau + \Delta \tau),\bar{\rho}^{N-1}_{\Phi,i}(\tau+\Delta \tau)))^2\\
&\quad - (\frac{1}{\widehat{\alpha}} \widehat{m}(\tau, y_i, \rho^N_{\Phi,i})\ \nabla_{v_i} {W}(\tau+\Delta \tau, v_i(\tau+\Delta t), \rho^{N-1}_{\Phi, i}(\tau+\Delta \tau)))^2  \Big|\\
& \quad    + \Big(   \frac{\widehat{\alpha}}{2} +\frac{\bar{\alpha}}{2}  \frac{N}{N^{\bar{a}}}  \Big)  \Big|\frac{C^2}{\widehat{\alpha}^2} \frac{N}{N^{\bar{\theta}}}  \nabla_{\bar{x}_k} \bar{W}(\tau+ \Delta \tau ,\bar{x}_i(\tau + \Delta \tau),\rho^{N-1}_{\Phi}(\boldsymbol{\bar{x}}_{-i}(\tau+\Delta \tau)))\\
&\quad  \nabla_{\bar{x}_i} \bar{W}(\tau+ \Delta \tau ,\bar{x}_i(\tau + \Delta \tau),\bar{\rho}^{N-1}_{\Phi,i}(\tau+\Delta \tau))   \Big|\\
&\quad + \Big(   \frac{\widehat{\alpha}}{2} +\frac{\bar{\alpha}}{2}  \frac{N}{N^{\bar{a}}}  \Big) \ \frac{C^2}{\widehat{\alpha}^2} \frac{N^2}{N^{2\ \bar{\theta}}}   \Big| \Big( \nabla_{\bar{x}_k} \bar{W}(\tau+ \Delta \tau ,\bar{x}_i(\tau + \Delta \tau),\rho^{N-1}_{\Phi}(\boldsymbol{\bar{x}}_{-i}(\tau+\Delta \tau)))   \Big)^2    \Big| \\
&\leq    \frac{C}{N^{\bar{\theta}}}+ \Big(   \frac{\widehat{\alpha}}{2} +\frac{\bar{\alpha}}{2}  \frac{N}{N^{\bar{a}}}  \Big) \frac{C}{\widehat{\alpha}} \Big| \nabla_{\bar{x}_i} \bar{W}(\tau+ \Delta \tau ,\bar{x}_i(\tau + \Delta \tau),\bar{\rho}^{N-1}_{\Phi,i}(\tau+\Delta \tau))\\
&\quad - \nabla_{v_i} {W}(\tau+\Delta \tau, v_i(\tau+\Delta t),  \rho^{N-1}_{\Phi, i}(\tau+\Delta \tau))\Big|\\
&\quad   \Big|  \nabla_{\bar{x}_i} \bar{W}(\tau+ \Delta \tau ,\bar{x}_i(\tau + \Delta \tau),\bar{\rho}^{N-1}_{\Phi,i}(\tau+\Delta \tau))+ \nabla_{v_i} {W}(\tau+\Delta \tau, v_i(\tau+\Delta \tau),  \rho^{N-1}_{\Phi, i}(\tau+\Delta \tau))  \Big| \\
&\quad + \Big(   \frac{\widehat{\alpha}}{2} +\frac{\bar{\alpha}}{2}  \frac{N}{N^{\bar{a}}}  \Big) \Big(   \frac{C^2}{\widehat{\alpha}^2} \frac{N}{N^{\bar{\theta}}}  \frac{C^2}{N}   + \ \frac{C^2}{\widehat{\alpha}^2} \frac{N^2}{N^{2\ \bar{\theta}}} \frac{C^2}{N^2}   \Big) \\
&\leq \frac{C}{N^{\bar{\theta}}}.
\end{align*}
It remains to show that the gradient satisfies the same growth conditions. Therefore we apply a Taylor approximation of zeroth order and get: 
\begin{align*}
&\Big| \nabla_y \big( \bar{W}(\tau, y,\bar{\rho}^{N-1}_{\Phi,i})-W(\tau, y, \rho^{N-1}_{\Phi, i}) \big)  \Big|   \\
&= \Big| \nabla_{2} \big( \bar{W}(\tau+\Delta \tau, \bar{x}_i(\tau+\Delta \tau),\bar{\rho}^{N-1}_{\Phi,i}(\tau+\Delta \tau) )-W(\tau+\Delta \tau, \bar{x}_i(\tau+\Delta \tau), \bar{\rho}^{N-1}_{\Phi,i}(\tau+\Delta \tau)) \big) \\
&\quad  \big(y-\bar{x}_i(\tau+\Delta \tau)\big)\\
&\quad  +  \partial_3  \big( \bar{W}(\tau+\Delta \tau, \bar{x}(\tau+\Delta \tau),\bar{\rho}^{N-1}_{\Phi, i} (\tau+\Delta \tau) )-W(\tau+\Delta \tau, \bar{x}(\tau+\Delta \tau), \bar{\rho}^{N-1}_{\Phi,i}(\tau+\Delta \tau)) \big) \\ &\quad \big(\rho^{N-1}_{\Phi, i}(\tau)-\bar{\rho}^{N-1}_{\Phi,i}(\tau+\Delta \tau)\big)  \Big|\\
&\leq \frac{C}{N^{\bar{\theta}}} \Delta \tau +  \frac{C}{N^{\bar{\theta}}}  | \Phi(y_k)- \Phi(x_k(\tau+\Delta \tau))  |        \leq \frac{C}{N^{\bar{\theta}}}.
\end{align*}

\end{proof}

\begin{remark}
The assumed rates for the difference of the value functions $W,\ \bar{W}$ in Corollary \ref{Connect} can be expected since the dynamics of $W$ and $\bar{W}$ only differ in the optimal control by a quantity of order $\mathcal{O}\Big( \frac{1}{N^{\bar{\theta}}}\Big)$.
\end{remark}

As next step we are ready to bound the error between the function $\langle W \rangle_N$ of \eqref{SymBellman} and the averaged value functions $\langle V_i \rangle_N$ of \eqref{ValueSys}. 

\begin{corollary}
Assume that unique solutions $V_i,\ i=1,...,N$ in the sense of Definition $\ref{GlobAs}$ of the system
\begin{subequations}\label{newV}
\begin{align}
\langle V_i(\tau, \boldsymbol{y})\rangle_N = & 	 - \Delta \tau  \Bigg\langle \Bigg( l(\tau,x_i(\tau), \rho^N_{\Phi} (\boldsymbol{x}(\tau)))+ \frac{\widehat{\alpha}}{2}\ (u_i^*(\tau))^2 +\frac{\bar{\alpha}}{2\ N^{\bar{a}}} \sum\limits_{k=1,\ k\neq i}^N\ (u_k^*(\tau))^2\Bigg)\Bigg\rangle_N\\
&+ \langle V_i(\tau+ \Delta \tau, \boldsymbol{x}(\tau+\Delta \tau))\rangle_N ,\\
u_i^*(\tau)= &\frac{1}{\widehat{\alpha}} \widehat{m}(\tau, x_i(\tau), \rho^N_{\Phi}(\boldsymbol{x}(\tau)))\ \nabla_{x_i} V_i(\tau+\Delta \tau, \boldsymbol{x}(\tau+\Delta \tau))\\
& +\frac{1}{\widehat{\alpha}\ N^{\bar{\theta}}}  \sum\limits_{k=1,\ k\neq i}^N \bar{m}(\tau, x_k(\tau), \rho^N_{\Phi}(\boldsymbol{x}(\tau))) \nabla_{x_k} V_i(\tau+\Delta \tau, \boldsymbol{x}(\tau+\Delta \tau)).
\end{align}
\end{subequations}
\begin{small}
\begin{align}\label{newD}
\begin{cases}
&{x}_i(\tau+\Delta \tau)= y_i+ \Delta \tau\ \Bigg[ f(\tau,x_i(\tau), \rho^N_{\Phi}(\boldsymbol{x}(\tau)))+ \widehat{m}(\tau,x_i(\tau), \rho^N_{\Phi}(\boldsymbol{x}(\tau)))\  u_i^*(\tau)  \\
&+ \frac{1}{N^{\bar{\theta}}}\bar{m}(\tau,x_i(\tau), \rho^N_{\Phi}(\tau))\ \sum\limits_{k=1,  k\neq i}^N u_k^*(\tau)\Bigg],\\
&x_i(\tau)=y_i. 
\end{cases}
\end{align}
\end{small}
exist. Additionally, we assume that the assumptions of Corollary \ref{Connect} are satisfied. 
Then the following inequalities 
\begin{align*}
&|\langle V_i(\tau, \boldsymbol{y} )-{W}(\tau, y_i, \rho^{N-1}_{\Phi, i}) \rangle_N|\leq \frac{C}{N^{\bar{\theta}}},\\
&|\langle \nabla_{y_i}V_i(\tau,\boldsymbol{y})-  \nabla_{y_i}{W}(\tau, y_i, \rho^{N-1}_{\Phi, i}) \rangle_N|\leq \frac{C}{N^{\bar{\theta}}},\\
& |\langle x_i(\tau+\Delta \tau)-v_i(\tau+\Delta \tau)\rangle_N| \leq \frac{C}{N^{\bar{\theta}}},
\end{align*}
 hold.
\end{corollary}

\begin{proof}
Thanks to the introduced scaling in \eqref{newV}- \eqref{newD} any solution $\langle \bar{W}(\tau,y_i,\bar{\rho}^{N-1}_{\Phi,i})\rangle_N$ can be identified by the function $\langle V_i\rangle_N$ . We define
$$
\langle V_i(\tau,\boldsymbol{y})\rangle_N := \langle \bar{W}(\tau,y_i,\bar{\rho}^{N-1}_{\Phi})\rangle_N  ,
$$
and immediately obtain that
\begin{align*}
&\langle \nabla_{y_i }V_i(\tau,\boldsymbol{y})\rangle_N =\langle \nabla_{y_i} \bar{W}(\tau,y_i,\bar{\rho}^{N-1}_{\Phi,i}) \rangle_N,\\
& \langle \nabla_{y_k }V_i(\tau,\boldsymbol{y})\rangle_N = \langle \nabla_{y_k} \bar{W}(\tau,y_i,\bar{\rho}^{N-1}_{\Phi}(\boldsymbol{y}_{-i})) \rangle_N= \Big\langle \partial_3  \bar{W}(\tau,y_i,\bar{\rho}^{N-1}_{\Phi,i}) \frac{\nabla_{y_k} \Phi(y_k) }{N} \Big\rangle_N,\ k\neq i
\end{align*}
holds. This identification is only possible because of the new scaling of the system \eqref{ValueSys}. 
For the difference between $\langle V_i \rangle_N$ and $\langle W \rangle_N$ we get: 
\begin{align*}
|\langle V_i(\tau, \boldsymbol{y})-W(\tau, y,  \rho^{N-1}_{\Phi, i})  \rangle_N| \leq &|\langle \bar{W}(\tau, y,  \bar{\rho}^{N-1}_{\Phi,i})-W(\tau, y, \rho^{N-1}_{\Phi, i})\rangle_N|\leq \frac{C}{N^{\bar{\theta}}}.
\end{align*}
Similarly, we obtain the following estimate for the state dynamics:
\begin{align*}
|x_i(\tau+\Delta \tau) - v_i(\tau+\Delta \tau)| =  |\bar{x}_i(\tau+\Delta \tau)- v_i(\tau+\Delta \tau)| \leq \frac{C}{N^{\bar{\theta}}}.
\end{align*}
For the gradient we are able to obtain the same estimate:
\begin{align*}
& |\langle \nabla_{y_i}V_i(\tau, \boldsymbol{y})-  \nabla_{y_i}{W}(\tau, y_i, \rho^{N-1}_{\Phi, i}) \rangle_N|\leq \frac{C}{N^{\bar{\theta}}}.
\end{align*}
\end{proof}
Notice that the value system \eqref{newV} considers the functions  $\boldsymbol{\bar{m}, \bar{\alpha}}$ and not the functions  $\bar{m}$ and $\bar{\alpha}$ 
of the original system \eqref{ValueSys}. 

\begin{remark}
In the case $\bar{\alpha}>1$  the quantity
$$
\frac{\bar{\alpha}}{2\ N^{\bar{a}}} \sum\limits_{k=1,\ k\neq i}^N (p_k^*(\tau))^2 \to 0,
$$
of equation \eqref{scaleA} vanishes as $N\to\infty$. Equivalently the quantity
$$
\frac{\bar{m}}{N^{\bar{\theta}}} \sum\limits_{k=1,\ k\neq i}^N p_k^*(\tau)\to 0,
$$
 of equation \eqref{scaleM} vanishes in the limit $N\to\infty$ provided that $\bar{\theta}>1$ holds. 

\end{remark}
For our further discussion we consider the system \eqref{SymBellman} and study the example $\bar{\alpha}=\bar{\theta}=1$. For other scalings the computations work in a similar way. Results are given in section \ref{SecScaling}. Since
$$
{W}(\tau+\Delta \tau, v_i(\tau+\Delta t),  \rho^{N-1}_{\Phi, i}(\tau+\Delta \tau)) = {W}(\tau, v_i(\tau),  \rho^{N-1}_{\Phi, i}(\tau))+\mathcal{O}(\Delta \tau),
$$
holds the state dynamics for arbitrary initial condition $x_i(\tau)$ are given by
\begin{subequations} \label{disDyn}
\begin{align}
&{z}_i(\tau+\Delta \tau)=y_i+ \Delta \tau \Bigg[  {f}(t,z_i(\tau), \rho^N_{\Phi,i}(\tau))\\
 &\quad\quad+ \widehat{{m}}(t,z_i(\tau),\rho^N_{\Phi,i}(\tau)) \left(\frac{1}{\widehat{\alpha}} \widehat{{m}}(\tau, z_i(\tau),\rho^N_{\Phi,i}(\tau))\ \nabla_{z_i} W(t, z_i(\tau),\rho^{N-1}_{\Phi, i}(\tau))   \right) \\
&\quad\quad + \bar{{m}}(t,z_i(\tau), \rho^N_{\Phi,i}(\tau))\ \frac{1}{N}\sum\limits_{k=1,  k\neq i}^N \frac{1}{\widehat{\alpha}} \widehat{{m}}(\tau, z_k(\tau),\rho^N_{\Phi,i}(\tau))\ \nabla_{z_k} W(t, z_k(\tau), \rho^{N-1}_{\Phi, i}(\tau)))\Bigg] ,\\
&z_i(\tau)=x_i(\tau).
\end{align}
\end{subequations}
 Thus, \eqref{disDyn} can be seen as an explicit Euler discretization of the corresponding time continuous dynamics. In comparison to the state dynamics \eqref{scaleM} we get
$$
z_i(\tau+\Delta \tau) =v_i(\tau+\Delta \tau)+\mathcal{O}((\Delta \tau)^2),
$$
provided that $x_i(\tau)=y$ holds. 
Thus, we can rewrite the discrete dynamics of the value function as 
\begin{footnotesize}
\begin{subequations}\label{FinalBell}
\begin{align}
 \Big\langle{W}(\tau, z_i(\tau),  &\rho^N_{\Phi,i }(\tau)) \Big\rangle_N =  	 - \Delta \tau \Bigg\langle \Bigg( {l}(\tau,z_i(\tau),\rho^N_{\Phi,i}(\tau))\\
& +\frac{\widehat{\alpha}}{2}\ \left( \frac{1}{\widehat{\alpha}} \widehat{{m}}(\tau, z_i(\tau), \rho^N_{\Phi,i}(\tau))\ \nabla_{z_i} {W}(\tau, z_i(\tau),\rho^N_{\Phi,i}(\tau)) \right)^2\\
&+\frac{\bar{\alpha}}{2} \frac{1}{N}\sum\limits_{k=1,\ k\neq i}^N\ \left(\frac{1}{\widehat{\alpha}} \widehat{{m}}(\tau, z_k(\tau),\rho^N_{\Phi,i}(\tau))\ \nabla_{z_k} {W}(\tau, z_k(\tau), \rho^N_{\Phi,i}(\tau)) \right)^2\Bigg) \\
&+{W}(\tau+ \Delta \tau,z_i(\tau+\Delta \tau), \rho^N_{\Phi,i}(\tau+\Delta \tau)) \Bigg\rangle_N . 
\end{align}
\end{subequations}
\end{footnotesize}
which is accurate up to $\mathcal{O}((\Delta \tau)^2)$ in comparison to \eqref{scaleA}. Finally, we are ready to discuss the mean field limit.

\paragraph{Mean Field Limit}

\begin{definition}
Given a vector $\bs{x}:=(x_1(,...,x_N)^{\top}\in \mathbb{R}^{dN},\ x_i:=(x_i^1,...,x_i^d)\in\R^d,\  i=1,...,N$ the empirical measure $\mu^N_{\bs{x}}$ is defined by
\[
\mu^N_{\bs{x}}(x^1,...,x^d):=\frac{1}{N} \sum\limits_{i=1}^N \delta(x^1-x_i^1)\cdot...\cdot \delta(x^d-x_i^d).
\]
\end{definition}
 We have for $\Phi$ as in Definition \ref{EmpDef} and the empirical measure $\mu^N_{\boldsymbol{x}}$
\begin{align*}
\rho^N_{\Phi}(\bs{x})&= \frac1N \sum\limits_{k=1}^N \Phi(x_k)= \int  \Phi(x^1,...,x^d)\ d\mu^N_{\bs{x}}( x^1,...,x^d)=: {\rho}_{\Phi}[\mu^N_{\bs{x}}]
\end{align*}
The computation shows that the empirical moment depends on the empirical measure $\mu^N_{\bs{x}}$. We apply this to functions $f, \widehat{m}, \bar{m}, l, q$, due to assumption \ref{ASym}. Exemplarily, we obtain for $f$
\begin{align*}
&f(\tau,x, \rho^N_{\Phi}(\boldsymbol{x}))=f\left(\tau,x,   {\rho}_{\Phi}[\mu^N_{\bs{x}}]\right). 
\end{align*}
 In the same manner, we rewrite the solution $\langle W \rangle _N$. 
 \begin{align}\label{ValueApp}
  \langle W(\tau,x_i,\rho^{N-1}_{\Phi, i}) \rangle_N= \langle {W}\left( \tau,x_i, {\rho}_{\Phi}[\mu^{N-1}_{\bs{x_{-i}}}]  \right) \rangle_N.
 \end{align}
 If $W$ is Lipschitz with respect to $\rho^{N-1}_{\Phi,i}$ and since
 $$
 ||\rho^{N}_{\Phi}-\rho^{N-1}_{\Phi, i} ||_{L_1(\R^{dN})} = \frac{\Phi(x_i)}{N},
 $$
holds, we have:
$$
\Big|\langle W(\tau,x_i,\rho^{N}_{\Phi})  -W(\tau,x_i,\rho^{N-1}_{\Phi, i})  \rangle_N \Big|  \leq \frac{C}{N}.
$$
Therefore, we replace in the dynamics \eqref{disDyn} $ W(\tau,z_i(\tau),\rho^{N-1}_{\Phi}(\boldsymbol{z}_{-i}(\tau)))$ by $W(\tau,z_i(\tau),\rho^{N}_{\Phi}(\boldsymbol{z}(\tau)))$. 
Then the time continuous limit of \eqref{disDyn} reads
 \begin{small}
 \begin{subequations}\label{MFlow}
 \begin{align}
 &\dot{z}_i(t)={f}(t,z_i(t),\rho^{N}_{\Phi}(t))\\
 &\quad\quad+ \widehat{{m}}(t,z_i(t), \rho^{N}_{\Phi}(t)) \left(\frac{1}{\widehat{\alpha}} \widehat{{m}}(t, z_i(t),\rho^{N}_{\Phi}(t))\ \nabla_{z_i} {W}(t, z_i(t), \rho^{N}_{\Phi}(t))   \right) \\
&\quad\quad + \bar{{m}}(t,z_i(t),\rho^{N}_{\Phi}(t))\ \frac{1}{N}\sum\limits_{k=1,  k\neq i}^N \frac{1}{\widehat{\alpha}} \widehat{{m}}(t, z_k(t),\rho^{N}_{\Phi}(t))\ \nabla_{z_k} {W}(t, z_k(t), \rho^{N}_{\Phi}(t)),\\
&z_i(\tau)=x_i(\tau),
 \end{align}
 \end{subequations}
 \end{small}
on $\tau\leq t \leq \tau+ \Delta \tau $ for arbitrary initial conditions $x_i(\tau)$. 
The corresponding discretized dynamic of the value function becomes
\begin{footnotesize}
\begin{subequations}
\begin{align}
\Big\langle {W}(\tau, z_i(\tau),  &\rho^{N}_{\Phi}(\tau)) \Big\rangle_N=  	 - \Delta \tau \Bigg\langle \Bigg( {l}(\tau,z_i(\tau), \rho^{N}_{\Phi}(\tau))\\
& +\frac{\widehat{\alpha}}{2}\ \left( \frac{1}{\widehat{\alpha}} \widehat{{m}}(\tau, z_i, \rho^{N}_{\Phi}(\tau))\ \nabla_{z_i} {W}(\tau, z_i(\tau),\rho^{N}_{\Phi}(\tau)) \right)^2\\
&+\frac{\bar{\alpha}}{2} \frac{1}{N}\sum\limits_{k=1,\ k\neq i}^N\ \left(\frac{1}{\widehat{\alpha}} \widehat{{m}}(\tau, z_k,\rho^{N}_{\Phi}(\tau))\ \nabla_{z_k} {W}(\tau, z_k(\tau), \rho^{N}_{\Phi}(\tau)) \right)^2\Bigg) \\
&+{W}(\tau+ \Delta \tau,z_i(\tau+\Delta \tau),\rho^{N}_{\Phi}(\tau+\Delta \tau)) \Bigg\rangle_N. 
\end{align}
\end{subequations}
\end{footnotesize}
We introduce the function $h: [0,T]\times \R^d \to \R$ defined by 
$$
h(\tau, z_i(\tau)):= {W}(\tau, z_i(\tau),  \rho_{\Phi}[\mu^N_{\boldsymbol{z}}](\tau)),
$$
for the empirical measure $\mu^N_{\boldsymbol{z}}(\tau)$ on $\R^d$.

\paragraph{Transport Equation}
Let $\phi(z), z\in \R^d$ be a test function.
\begin{align*}
\frac{d}{dt} &\int\limits_{\R} \phi(z)\ \ d\mu^{N}_{\boldsymbol{z}}(t,z) = \frac{1}{N} \sum\limits_{i=1}^N \nabla_z \phi(z)\ \Bigg(  {f}(t,z_i(t), \rho^{N}_{\Phi}(t))\\
&+ \widehat{{m}}(t,z_i(t), \rho^{N}_{\Phi}(t)) \left(\frac{1}{\widehat{\alpha}} \widehat{{m}}(t, z_i(t),  \rho^{N}_{\Phi}(t))\ \nabla_{z_i} h(t,z_i(t))   \right) \\
& + \bar{{m}}(t,z_i(t), \rho^{N}_{\Phi}(t))\ \frac{1}{N}\sum\limits_{k=1,  k\neq i}^N \frac{1}{\widehat{\alpha}} \widehat{{m}}(t, z_k(t),  \rho^{N}_{\Phi})\ \nabla_{z_k} h(t, z_k(t))   \Bigg) \\
\stackrel{\eqref{approxEmp}}{=}& \int\limits_{\R}  \nabla_z \phi(z)\ \Bigg(  {f}(t,z, \rho^{N}_{\Phi}(t))+ \widehat{{m}}(t,z,  \rho^{N}_{\Phi}(t)) \left(\frac{1}{\widehat{\alpha}} \widehat{{m}}(t, z,  \rho^{N}_{\Phi}(t))\ \nabla_{z} h(t, z)   \right) \\
& + \bar{{m}}(t,z, \rho^{N}_{\Phi}(t))\  \int\limits_{\R} \frac{1}{\widehat{\alpha}} \widehat{{m}}(t, z^{\prime},  \rho^{N}_{\Phi}(t))\ \nabla_{z^{\prime}} h(t, z^{\prime})\  d\mu^N_{\boldsymbol{z}}(t,z^{\prime})\\
&- \frac1N  \bar{{m}}(t,z, \rho^{N}_{\Phi}(t))\  \int\limits_{\R} \frac{1}{\widehat{\alpha}} \widehat{{m}}(t, \tilde{z},  \rho^{N}_{\Phi}(t))\ \nabla_{z^{\prime}} h(t, \tilde{z})\  d\mu^1_{z_i}      \Bigg)\ d\mu^N_{\boldsymbol{z}}(t,z)\  
 \end{align*}
 Here, we have used that
  \begin{subequations}\label{approxEmp}
 \begin{align}
\frac{1}{N}\sum\limits_{k=1 }^N \frac{1}{\widehat{\alpha}} \widehat{{m}}&(t, z_k(t),  \rho^{N}_{\Phi}(t))\ \nabla_{z_k} h(t, z_k(t))\\
&-  \frac{1}{N}\sum\limits_{k=1,  k\neq i}^N \frac{1}{\widehat{\alpha}} \widehat{{m}}(t, z_k(t), \rho^{N}_{\Phi}(t))\ \nabla_{z_k} h(t, z_k(t))\\
 &= \frac{1}{N}  \frac{1}{\widehat{\alpha}} \widehat{{m}}(t, z_i(t), \rho^{N}_{\Phi}(t))\ \nabla_{z_i}h(t, z_i(t)),
 \end{align}
 \end{subequations}

\paragraph{HJB Equation} In the following we derive the corresponding HJB equation of equation \eqref{FinalBell}.
As before we employ the empirical measure to rewrite averages as integrals. Thus, we obtain:
\begin{subequations}\label{closeTo}
\begin{footnotesize}
\begin{align}
& \int\limits_{\R^d} h(\tau, \tilde{z})\ d\mu^N_{\boldsymbol{z}}(\tau,\tilde{z})-  \int\limits_{\R^d} h(\tau+ \Delta \tau,\tilde{z})  \ d\mu^N_{\boldsymbol{z}}(\tau+\Delta \tau,\tilde{z})\\
&\stackrel{\eqref{secondApro}}{=}- \Delta \tau   \int\limits_{\R^d}   {l}(\tau,\tilde{z},  \rho^{N}_{\Phi}(\tau))+ \frac{\widehat{\alpha}}{2}\ \left( \frac{1}{\widehat{\alpha}} \widehat{{m}}(\tau,\tilde{z}, \rho^{N}_{\Phi}(\tau))\ \nabla_{z}h(\tau,\tilde{ z}) \right)^2 d\mu^N_{\bs{z}}(\tau, \tilde{z}) \\
&\quad\quad- \Delta \tau \frac{\bar{\alpha}}{2} \int\limits_{\R^d} \int\limits_{\R^d} \left(\frac{1}{\widehat{\alpha}} \widehat{{m}}(\tau, z^{\prime}, \rho^{N}_{\Phi}(\tau))\ \nabla_{z^{\prime}} h(\tau, z^{\prime}) \right)^2\   d\mu^{N}_{\boldsymbol{z}}(\tau,z^{\prime})\ d\mu^N_{\boldsymbol{z}}(\tau,\tilde{z})\\
&\quad\quad+ \Delta \tau \frac{1}{N} \frac{\bar{\alpha}}{2} \int\limits_{\R^d} \int\limits_{\R^d} \left(\frac{1}{\widehat{\alpha}} \widehat{{m}}(\tau, z^{\prime}, \rho^{N}_{\Phi}(\tau))\ \nabla_{z^{\prime}} h(\tau, z^{\prime}) \right)^2\  d\mu^1_{z_i}  \ d\mu^N_{\boldsymbol{z}}(\tau, \tilde{z}). 
\end{align}
\end{footnotesize}
\end{subequations}
Here, we have additionally used:
\begin{subequations}\label{secondApro}
\begin{align}
\frac{\bar{\alpha}}{2} \frac{1}{N}\sum\limits_{k=1}^N\ & \left(\frac{1}{\widehat{\alpha}} \widehat{{m}}(\tau, z_k,  \rho^{N}_{\Phi}(\tau))\ \nabla_{z_k} h(\tau, z_k(\tau)) \right)^2\\
&-  \frac{\bar{\alpha}}{2} \frac{1}{N}\sum\limits_{k=1,\ k\neq i}^N\ \left(\frac{1}{\widehat{\alpha}} \widehat{{m}}(\tau, z_k,  \rho^{N}_{\Phi}(\tau))\ \nabla_{z_k}h(\tau, z_k(\tau)) \right)^2\\
&= \frac{1}{N}  \frac{\bar{\alpha}}{2}  \left(\frac{1}{\widehat{\alpha}} \widehat{{m}}(\tau, z_i, \rho^{N}_{\Phi}(\tau))\ \nabla_{z_i} h(\tau, z_i(\tau)) \right)^2
\end{align}
\end{subequations}
As next step we divide by $-\Delta \tau$ and consider the limit $\Delta \tau\to 0$. Thus, we formally obtain a temporal derivative on the left hand side of \eqref{closeTo}. 
\begin{subequations}\label{FinalHJB}
\begin{footnotesize}
\begin{align}
& \int\limits_{\R^d} \frac{d}{d\tau}  h(\tau, \tilde{z})\ d\mu^N_{\boldsymbol{z}}(\tau,\tilde{z})=\\
&  \int\limits_{\R^d}   \Bigg( {l}(\tau,\tilde{z}, \rho^{N}_{\Phi}(\tau))+ \frac{\widehat{\alpha}}{2}\ \left( \frac{1}{\widehat{\alpha}} \widehat{{m}}(\tau, \tilde{z},  \rho^{N}_{\Phi}(\tau))\ \nabla_{\tilde{z}}h(\tau, \tilde{z}) \right)^2\\
&\quad\quad+\frac{\bar{\alpha}}{2} \int\limits_{\R^d} \left(\frac{1}{\widehat{\alpha}} \widehat{{m}}(\tau, z^{\prime},  \rho^{N}_{\Phi}(\tau))\ \nabla_{z^{\prime}} h(\tau, z^{\prime}) \right)^2   d\mu^{N}_{\boldsymbol{z}}(\tau, z^{\prime})\Bigg) \ d\mu^N_{\boldsymbol{z}}(\tau,\tilde{z})\\
&\quad\quad- \frac{1}{N} \frac{\bar{\alpha}}{2} \int\limits_{\R^d}  \int\limits_{\R^d} \left(\frac{1}{\widehat{\alpha}} \widehat{{m}}(\tau, z^{\prime}, \rho^{N}_{\Phi}(\tau))\ \nabla_{z^{\prime}} h(\tau, z^{\prime}) \right)^2  \ d\mu^1_{z_i}  \ d\mu^N_{\boldsymbol{z}}(\tau,\tilde{z}).
\end{align}
\end{footnotesize}
\end{subequations}
Note that the following inequalities
 \begin{footnotesize}
 \begin{align*}
&\Big|  \frac{1}{N} \int\limits_{\R} \left(\frac{1}{\widehat{\alpha}} \widehat{{m}}(\tau, z^{\prime}, \mu^N_{\boldsymbol{z}}(\tau))\ \nabla_{z^{\prime}} {W}(\tau, z^{\prime}, \rho^N_{\Phi}[\mu^{N}_{\boldsymbol{z}}](\tau)) \right)^2 \ d\mu^1_{z_i}\Big|\leq \frac{C}{N},\\
&\Big|  \frac{1}{N} \int\limits_{\R} \frac{1}{\widehat{\alpha}} \widehat{{m}}(\tau, z^{\prime}, \mu^N_{\boldsymbol{z}}(\tau))\ \nabla_{z^{\prime}} {W}(\tau, z^{\prime}, \rho^N_{\Phi}[\mu^{N}_{\boldsymbol{z}}](\tau))   \ d\mu^1_{z_i}  \Big|\leq  \frac{C}{N},
 \end{align*}
 \end{footnotesize}
hold, since 
$$
|\nabla_{z^{\prime}}  (\tau, z^{\prime}, \rho^N_{\Phi}[\mu^{N}_{\boldsymbol{z}}](\tau))  | \leq C
$$
for all $\tau>0, z^{\prime}\in\R^d,\  \rho^N_{\Phi}\in\R$. Thus, both terms are of order $\mathcal{O}(\frac1N)$ and we can neglect them in the limit. Thus, the HJB and transport equation are given up to order  $\mathcal{O}(\frac1N)$ by:

 \begin{subequations}\label{weakTrans}
\begin{align}
\frac{d}{dt} &\int\limits_{\R} \phi(z)\ \ d\mu^{N}_{\boldsymbol{z}}(t,z) =\\
&  \int\limits_{\R}  \nabla_z \phi(z)\ \Bigg(  {f}(t,z, \rho^{N}_{\Phi}(t))+ \widehat{{m}}(t,z,  \rho^{N}_{\Phi}(t)) \left(\frac{1}{\widehat{\alpha}} \widehat{{m}}(t, z,  \rho^{N}_{\Phi}(t))\ \nabla_{z} h(t, z)   \right) \\
& \quad+ \bar{{m}}(t,z, \rho^{N}_{\Phi}(t))\  \int\limits_{\R} \frac{1}{\widehat{\alpha}} \widehat{{m}}(t, z^{\prime},  \rho^{N}_{\Phi}(t))\ \nabla_{z^{\prime}} h(t, z^{\prime})\  d\mu^N_{\boldsymbol{z}}(t,z^{\prime})       \Bigg)\ d\mu^N_{\boldsymbol{z}}(t,z).
\end{align}
\end{subequations}
Thus, equation \eqref{weakTrans} is the weak form of the following transport equation:
\begin{subequations}\label{MFTransport}
\begin{align}
\partial_t & \mu^{N}_{\boldsymbol{z}}(t,z) + div _z \Bigg[ \Bigg(  {f}(t,z,\rho^{N}_{\Phi}(t) )+ \widehat{{m}}(t,z,\rho^{N}_{\Phi}(t)) \left(\frac{1}{\widehat{\alpha}} \widehat{{m}}(t, z, \rho^{N}_{\Phi}(t))\ \nabla_{z} h(t, z)   \right) \\
&+ \bar{{m}}(t,z, \rho^{N}_{\Phi}(t))\  \int\limits_{\R} \frac{1}{\widehat{\alpha}} \widehat{{m}}(t, z^{\prime},\rho^{N}_{\Phi}(t))\ \nabla_{z^{\prime}} h(t, z^{\prime})\  \mu^N_{\boldsymbol{z}}(t,z^{\prime})    dz^{\prime} \Bigg)\mu^{N}_{\boldsymbol{z}}(t,z)\Bigg] = 0.
\end{align}
\end{subequations}
Equivalently,  the HJB equation \eqref{FinalHJB} up to order $\mathcal{O}(\frac1N)$ reads

\begin{subequations}
\begin{footnotesize}
\begin{align}
  \frac{d}{d\tau} \int\limits_{\R^d} & h(\tau, \tilde{z})\ d\mu^N_{\boldsymbol{z}}(\tau,\tilde{z})=
 \int\limits_{\R^d}   \Bigg( {l}(\tau,\tilde{z}, \rho^{N}_{\Phi}(\tau))+ \frac{\widehat{\alpha}}{2}\ \left( \frac{1}{\widehat{\alpha}} \widehat{{m}}(\tau, \tilde{z},  \rho^{N}_{\Phi}(\tau))\ \nabla_{\tilde{z}}h(\tau, \tilde{z}) \right)^2\\
&+\frac{\bar{\alpha}}{2} \int\limits_{\R^d} \left(\frac{1}{\widehat{\alpha}} \widehat{{m}}(\tau, z^{\prime},  \rho^{N}_{\Phi}(\tau))\ \nabla_{z^{\prime}} h(\tau, z^{\prime}) \right)^2   d\mu^{N}_{\boldsymbol{z}}(\tau, z^{\prime}) \Bigg) \ d\mu^N_{\boldsymbol{z}}(\tau,\tilde{z}).
\end{align}
\end{footnotesize}
\end{subequations}
This is equivalent to 
\begin{align*}
 \int\limits_{\R}&  \partial_{\tau } h(\tau,z)\  d\mu^N_{\boldsymbol{z}}(\tau) + \int\limits_{\R}  h(\tau,z)\ \partial_{\tau} d\mu^N_{\boldsymbol{z}}(\tau,z)\\
 &=  \int\limits_{\R}   \Bigg( {l}(\tau,z, \rho^{N}_{\Phi} )+ \frac{\widehat{\alpha}}{2}\ \left( \frac{1}{\widehat{\alpha}} \widehat{{m}}(\tau, z,\rho^{N}_{\Phi})\ \nabla_{z} h(\tau,z)\right)^2  \\
 &\quad \quad + \frac{\bar{\alpha}}{2} \int\limits_{\R} \left(\frac{1}{\widehat{\alpha}} \widehat{{m}}(\tau, z^{\prime},\rho^{N}_{\Phi})\ \nabla_{z^{\prime}} h(\tau,z^{\prime})) \right)^2  d\mu^N_{\boldsymbol{z}}(t,z^{\prime}) \Bigg) \ d\mu^N_{\boldsymbol{z}}(\tau,z). 
\end{align*}
Then we use the transport equation \eqref{MFTransport}  and get
\begin{small}
\begin{align*}
\int\limits_{\R}  &\  \partial_{\tau } h(\tau,z)\ d\mu^N_{\boldsymbol{z}}(\tau,z)-  \int\limits_{\R} h(\tau,z)\ div_{z} \Big(\Big[{f}(\tau,z,\rho^{N}_{\Phi})+\frac{1}{\widehat{\alpha}}(\widehat{{m}}(\tau,z,\rho^{N}_{\Phi}))^2\ \nabla_z h(\tau,z)\\
&+ \frac{1}{\widehat{\alpha}} \bar{{m}}(\tau,z,\rho^{N}_{\Phi})\ \int\  \nabla_{z^{\prime}}  h(\tau,z^{\prime})\ \widehat{{m}}(\tau,z^{\prime},\rho^{N}_{\Phi})\ d \mu^N_{\boldsymbol{x}}(\tau,z^{\prime})\Big]\ \mu^N_{\boldsymbol{z}}(\tau,z) \Big)  dz\\
 &=  \int\limits_{\R}   \Bigg( {l}(\tau,z, \rho^{N}_{\Phi})+ \frac{\widehat{\alpha}}{2}\ \left( \frac{1}{\widehat{\alpha}} \widehat{{m}}(\tau, z,\rho^{N}_{\Phi})\ \nabla_{z} h(\tau,z)\right)^2\\
 &\quad\quad \quad +\frac{\bar{\alpha}}{2} \int\limits_{\R} \left(\frac{1}{\widehat{\alpha}} \widehat{{m}}(\tau, z^{\prime},\rho^{N}_{\Phi})\ \nabla_{z^{\prime}} h(\tau,z^{\prime})) \right)^2  d\mu^N_{\boldsymbol{z}}(t,z^{\prime}) \Bigg) \ d\mu^N_{\boldsymbol{z}}(\tau,z).
\end{align*}
\end{small}
Integration by parts gives then the final HJB equation 
\begin{small}
\begin{align*}
\int\limits_{\R}  &\Bigg(   \partial_{\tau } h(\tau,z)+  \nabla_z h(\tau,z)\ \Big[{f}(\tau,z,\rho^{N}_{\Phi})+\frac{1}{\widehat{\alpha}}(\widehat{{m}}(\tau,z,\rho^{N}_{\Phi}))^2\ \nabla_z h(\tau,z)\\
&+ \frac{1}{\widehat{\alpha}} \bar{{m}}(\tau,z,g)\ \int \nabla_{z^{\prime}}  h(\tau,z^{\prime})\ \widehat{{m}}(\tau,z^{\prime},\rho^{N}_{\Phi})\ d \mu^N_{\boldsymbol{z}}(\tau,z^{\prime})\Big] \Bigg)\  d\mu^N_{\boldsymbol{z}}(\tau,z)\\
 &=  \int\limits_{\R}   \Bigg( {l}(\tau,z, \rho^{N}_{\Phi})+ \frac{\widehat{\alpha}}{2}\ \left( \frac{1}{\widehat{\alpha}} \widehat{{m}}(\tau, z,\rho^{N}_{\Phi})\ \nabla_{z} h(\tau,z)\right)^2\\
 &\quad\quad \quad +\frac{\bar{\alpha}}{2} \int\limits_{\R} \left(\frac{1}{\widehat{\alpha}} \widehat{{m}}(\tau, z^{\prime},\rho^{N}_{\Phi})\ \nabla_{z^{\prime}} h(\tau,z^{\prime})) \right)^2 \ d\mu^N_{\boldsymbol{z}}(t,z^{\prime}) \Bigg) \ d\mu^N_{\boldsymbol{z}}(\tau,z).
\end{align*}
\end{small}
tested against the empirical measure  $\mu^N_{\boldsymbol{z}}\geq 0$. \\
We assume that in the limit $N\to\infty$ the empirical measure $\mu^N_{\boldsymbol{z}}$ converges to a measure with probability distribution function $g: [0,T]\times\R^d\to \R$. 
Thus, the mean field game limit equations are given on the support of $g$ and for all $t\in[0,T], x\in\R^d$ by:
 \begin{footnotesize}
 \begin{align*}
&\partial_t h(t,x)+\nabla_{x} h(t, x)\  \left( {f}(t,x,\rho_{\Phi}[g])+\frac{1}{\widehat{\alpha}} \bar{{m}}(t,x,\rho_{\Phi}[g])\  \int g(t,z)\ \widehat{{m}}(t,z,\rho_{\Phi}[g])\  \nabla_{z} h(t, z) \ dz \right)\\
 &\quad= {l}(t,x,\rho_{\Phi}[g]) -\frac{1}{2\ \widehat{\alpha}} \Big[\nabla_{x} h(t,x)\ \widehat{{m}}(t,x,\rho_{\Phi}[g])\Big]^2 + \frac{\bar{\alpha}}{2\ \widehat{\alpha}^2}\ \int g(t,z)\ \big(\nabla_{z} h(t,z)\ \widehat{{m}}(t,z,\rho_{\Phi}[g])\big)^2\ dz,\\
 &\partial_t g(t,x)+ div_{x} \Big(\Big[{f}(t,x,g)+\frac{1}{\widehat{\alpha}}(\widehat{{m}}(t,x,\rho_{\Phi}[g]))^2\ \nabla_x h(t,x)\\
 &\quad\quad\quad\quad \quad\quad + \frac{1}{\widehat{\alpha}} \bar{{m}}(t,x,\rho_{\Phi}[g]\ \int g(t,z)\  \nabla_{z}  h(t,z)\ \widehat{{m}}(t,z,\rho_{\Phi}[g])\ dz\Big]\ g(t,x) \Big)=0,\\
 &g(0,x)=g_0,\quad h(T,x)={p}(x,\rho_{\Phi}[g](T)). 
\end{align*} 
\end{footnotesize}
Here, we have replaced the dependence on the empirical moment $\rho^N_{\Phi}$ by the corresponding moment of the probability distribution function $g$ given by $ \rho_{\Phi}[g](t):= \int \Phi(x)\ g(t,x)\ dx$. 
The previous MFG limit system corresponds to the choice of the \textbf{non-local-coupling}. The limit equations of the other couplings are easily obtained by different scalings with respect to the number of agents. 

\subsection{ Scalings for $\widehat{m}$ and $\widehat{\alpha}$}\label{SecScaling}
In the previous section we have seen that the scaling of the quantities $\boldsymbol{\bar{\alpha}}$ and $\boldsymbol{\bar{m}}$ have to satisfy
$$\boldsymbol{\bar{\alpha}}\sim \mathcal{O}\left(\frac{1}{N^{\bar{a}}}\right), \bar{a}\geq 1\quad \boldsymbol{\bar{m}}\sim \mathcal{O}\left(\frac{1}{N^{\bar{\theta}}}\right),\ \bar{\theta}\geq 1.$$\\
This choice was necessary in order to obtain a closed equation for $W$. Therefore, we aim to investigate alternative scalings for the functions $\widehat{m}$ and weight $\widehat{\alpha}$.
We assume the following scaling of value function $W$
\begin{align}
\nabla_{y_i} W\Big(t,y_i,\rho^{N-1}_{\Phi}( \boldsymbol{y}_{-i})\Big)\sim\mathcal{O}(1).\label{Ass1}
\end{align}
 For the derivative with respect to a symmetric variable $j,\ j\neq i$ we get:
\begin{align}
\nabla_{y_j} W\Big(t,y_i, \rho^{N-1}_{\Phi}( \boldsymbol{y}_{-i})\Big)\sim\mathcal{O}\left(\frac{1}{N}\right).\label{Ass2}
\end{align}

\begin{lemma}\label{ScaleTheo}
We assume that  
$$
\boldsymbol{\bar{\alpha}}\sim \mathcal{O}(\frac{1}{N^{\bar{a}}}),\ \bar{\alpha}\geq 1\quad \boldsymbol{\bar{m}}\sim \mathcal{O}(\frac{1}{N^{\bar{\theta}}}),\ \bar{\theta}\geq 1,\ f\sim \mathcal{O}(1),\quad l\sim\mathcal{O}(1),\quad p\sim \mathcal{O}(1)
$$
 and equations \eqref{Ass1},\eqref{Ass2} hold. 
We rescale the quantities $\widehat{m}$ and $\widehat{\alpha}$ with respect to the number of agents $N$; we  consider  
\begin{align*}
&\boldsymbol{\widehat{m}} := \frac{1}{N^{\widehat{\theta}}} \widehat{m },\quad \boldsymbol{\widehat{\alpha}}:=\frac{1}{N^{\widehat{a}}} \widehat{\alpha},
\end{align*}
for $\widehat{\theta},\widehat{a}\geq 0$. Then the set of all possible scalings are an interplay of the scaling parameters $\widehat{\theta}, \bar{\theta},\widehat{a}, \bar{a} $ defined by the following inequalities.
\begin{align*}
&i)\ \quad \widehat{a}-2\ \widehat{\theta} \leq  0,\\
&ii)\ \quad  \widehat{a}+1-\widehat{\theta}-\bar{\theta} \leq  0,\\
&iii)\ \quad 2\ \widehat{a}+1-2\ \widehat{\theta}-\bar{a} \leq 0.
\end{align*}

\end{lemma}
\begin{proof}
First, we analyze the scales of the quantities $\boldsymbol{f}_i$ and $L_i$ in the system \eqref{SymBellman}.
By assumption  $f\sim \mathcal{O}(1), p \sim \mathcal{O}(1)$ and $l\sim \mathcal{O}(1)$ holds. Thus, the remaining sums of $\boldsymbol{f}_i$ asymptotically satisfy:
\begin{align*}
&i) \quad  \boldsymbol{\widehat{m}} \left(\frac{1}{\boldsymbol{\widehat{\alpha}}} \boldsymbol{\widehat{m}}\ \nabla_{x_i} W(x_i, \rho^{N-1}_{\Phi}(\boldsymbol{x}_{-i}))  \right) \sim \frac{\boldsymbol{\widehat{m}}^2}{\boldsymbol{\widehat{\alpha}}} \ \mathcal{O}(1) ,\\
& ii) \quad \frac{1}{N^{\bar{\theta}}}\bar{m}\ \sum\limits_{k=1,  k\neq i}^N \frac{1}{\boldsymbol{\widehat{\alpha}}} \boldsymbol{\widehat{m}}\ \nabla_{x_k} W(x_k, \rho^{N-1}_{\Phi}(\boldsymbol{x}_{-k})) \sim  \frac{\boldsymbol{\widehat{m}}}{\boldsymbol{\widehat{\alpha}}}\ \mathcal{O}\Big(    \frac{N}{N^{\bar{\theta}}}  \Big).
\end{align*}
Respectively, the remaining terms of $L_i$ satisfy:
\begin{align*}
&iii)\quad\frac{1}{N^{\bar{a}}}  \frac{\bar{\alpha}}{2} \sum\limits_{k=1,\ k\neq i}^N\ \left( \frac{1}{\widehat{\alpha}} \widehat{m}\ \nabla_{x_k} W(x_k, \boldsymbol{x}_{-k}) \right) ^2\sim\ \mathcal{O}\Big( \frac{N}{N^{\bar{a}}} \Big)  \frac{\boldsymbol{\widehat{m}}^2}{\boldsymbol{\widehat{\alpha}}^2}, \\
&iv)\quad  \frac{\boldsymbol{\widehat{\alpha}}}{2}\ \left( \frac{1}{\boldsymbol{\widehat{\alpha}}} \boldsymbol{\widehat{m}}\ \nabla_{x_i} W(x_i, \rho^{N-1}_{\Phi}(\boldsymbol{x}_{-i})) \right)^2 \sim \frac{\boldsymbol{\widehat{m}}^2}{\boldsymbol{\widehat{\alpha}}} \mathcal{O}(1).
\end{align*}
The scales $i)$ and $iv)$ are asymptotically identical. Hence, the relevant quantities read:
\begin{align*}
&i)\ \frac{\boldsymbol{\widehat{m}}^2}{\boldsymbol{\widehat{\alpha}}}  \mathcal{O}(1)\sim \mathcal{O}\Big( \frac{N^{\widehat{a}}}{N^{2\widehat{\theta}}} \Big),\quad \widehat{a}-2\ \widehat{\theta}  \leq 0,\\
&ii)\ \frac{\boldsymbol{\widehat{m}}}{\boldsymbol{\widehat{\alpha}}}\ \mathcal{O}\Big(  \frac{N}{N^{\bar{\theta}}}\Big) \sim   \mathcal{O}\Big( \frac{N^{\widehat{a}+1}}{N^{\widehat{\theta}+\bar{\theta}}}\Big),\quad  \widehat{a}+1-\widehat{\theta}-\bar{\theta} \leq 0,\\
&iii) \frac{\boldsymbol{\widehat{m}}^2}{\boldsymbol{\widehat{\alpha}}^2}  \mathcal{O}\Big( \frac{N}{N^{\bar{a}}} \Big) \sim  \mathcal{O}\Big(  \frac{N^{2\ \widehat{a}+1}}{N^{2\ \bar{\theta}+ \bar{a}}}\Big),\quad 2\ \widehat{a}+1-2\ \bar{\theta}-\bar{a}\leq 0.
\end{align*}
\end{proof}
The inequalities $i)-iii)$ then define the precise form of the mean field game limit equations stated previously \eqref{MFGFirst}.

\begin{remark}
The scaling considered in the previous section in order to derive the mean field limit is given by:
$$
\bar{a}=\bar{\theta}=1,\quad \widehat{a}=\widehat{\theta}=0. 
$$
Clearly, for this choice we have equality in the quantities $i)-iii)$.  
\end{remark}

\section{ Financial Market Model}

In this section, we give an explicit example of a \textbf{non-local-coupling} which has been introduced in \cite{trimborn2017kinetic}. Inspired by the econophysical Levy-Levy-Solomon model \cite{levy2000microscopic}, we consider $N$ financial agents. 
Each agent is equipped with two portfolios, one portfolio represents the investment in a risky stock, the other the investment in safe bonds. The sum of the risky investments $x\in\R_{\geq0}$ and the risk-free investments $y\in\R_{\geq 0}$ is the overall wealth of the i-th agent $w_i:=x_i+y_i,\ i=1,...,N$. The system reads:
\begin{subequations}\label{exampleMicro}
\begin{align}
&\dot{x}_i=\kappa\ \frac{\dot{S}+D}{S}\ x_i+u_i^*,\\
&\dot{y}_i=r\ y_i-u_i^*,\\
&S:= \lambda\ \frac{1}{N} \sum\limits_{k=1}^N x_k,\\
&u_i^*=\argmax\limits_{u_i\in \R} - \int\limits_t^T   l(s,x_i, \rho_{\Phi}^N(\boldsymbol{x}), y_i, \rho_{\Phi}^N(\boldsymbol{y}))+ \frac{\widehat{\alpha}}{2}\ u_i^2\ ds.
\end{align}
\end{subequations}
Here, $\boldsymbol{x}:= (x_1,...,x_N)^{\top}\in\R^N$ and  $r,D>0,\  \lambda,\kappa\in (0,1)$ are positive constants.
We denote by $S\in\R_{\geq 0}$ the stock price where $D$ symbolizes a dividend. Furthermore, is $r$ the interest rate of the safe asset, $\lambda$ the market depth and $\kappa$ expresses transaction costs.
The agents can shift their money between both portfolios by the optimal control $u_i^*$, a positive control corresponds to a shift of the money from the bond portfolio to the stock portfolio. 
We assume quadratic costs and the objective function $l$ is not specified, but is only allowed to depend on $t,x_i,y_i$ and some empirical moment $\rho^N_{\Phi}$. \\[1 em]
Before we can derive the limit system, we need an explicit ODE for the risky portfolio. We define $e:=(1,...,1)^{\top}\in\R^N,\ I:= diag(1)\in\R^{N\times N}$ and rewrite our risky asset equation:
\begin{align*}
\dot{\boldsymbol{x}}&= \diag(\kappa)\ \frac{\frac{\lambda}{N} e^{\top}\dot{\boldsymbol{x}}+D}{\frac{\lambda}{N}  e^{\top} \boldsymbol{x} } \boldsymbol{x}+ \boldsymbol{u}\\
\Big( I-\underbrace{\diag(\kappa) \frac{\boldsymbol{x}\ e^{\top} }{e^{\top} \boldsymbol{x}}}_{=:P(\textbf{x})}\Big) \dot{\boldsymbol{x}}&= \diag{\kappa} \frac{D}{\frac{\lambda}{N} e^{\top}\boldsymbol{x}} \boldsymbol{x}+\boldsymbol{u}
\end{align*}
The matrix $P$ is a rank one matrix and $||P||_1=\kappa<0$ holds. Hence, the inverse of $\Sigma(\boldsymbol{x}):=I-P(\boldsymbol{x})$ exists and has a Neumann series expansion. We get:
\[
\Sigma^{-1}(\boldsymbol{x})=\sum\limits_{k=0}^N P^k(\boldsymbol{x})=I+\sum\limits_{k=0}^N \alpha^k P(\boldsymbol{x})=I+\frac{1}{1-\alpha} P(\boldsymbol{x}),\quad \alpha:=trace(P(\boldsymbol{x}))=\kappa.
\]
Thus, the explicit stock ODE is given by:
\[
\dot{\boldsymbol{x}}= \frac{D}{\frac{\lambda}{N} e^{\top}\boldsymbol{x}}\ \Sigma^{-1}(\boldsymbol{x})\ \diag(\kappa)\ \boldsymbol{x}+ \Sigma^{-1}(\boldsymbol{x})\ \boldsymbol{u}.
\]
For the i-th agent we observe:
\begin{align*}
&\dot{x}_i=\big( c_1+ c_2 \frac{1}{N} \sum\limits_{k=1}^N u_k^*  \big)\ \frac{x_i}{\frac{1}{N}\sum\limits_{k=1}^Nx_k }+u_i^*,\\
&\dot{y}_i = r\ y_i-u_i^*,\\
&u_i^*=\argmax\limits_{u_i\in \R} - \int\limits_t^T   l(s,x_i, \rho_{\Phi}^N(\boldsymbol{x}), y_i, \rho_{\Phi}^N(\boldsymbol{y}))+ \frac{\widehat{\alpha}}{2}\ u_i^2\ ds,
\end{align*}
where the constants $c_1,c_2$ are defined by: $c_1:=(1+\frac{\kappa}{1-\kappa})\ \kappa\ \frac{D}{\lambda},\ c_2:= \frac{\kappa}{1-\kappa}$.
We consider an arbitrary model for $l$, which satisfies the symmetry assumption \ref{ASym}.
This model fits into the previously introduced framework.
\begin{lemma}
 We verify the symmetry assumptions and scaling properties of the microscopic system \eqref{exampleMicro}.  
\begin{align*}
&f(x_i,\rho^N(\boldsymbol{x}),y_i)=\Big( c_1\ \frac{x_i}{\rho^N(\boldsymbol{x}) }, r\ y_i\Big)^{\top}\sim   \mathcal{O}(1),\quad \rho^N(\boldsymbol{x}):=\frac{1}{N} \sum\limits_{k=1}^N  x_k,\\
& M_i^i=\widehat{m}= \left( c_2\ \frac{x_i}{\rho^N(\boldsymbol{x}) } \frac{1}{N}+1, -1\right)^{\top}\sim \mathcal{O}(1),\\                       
 & M^i_{k,k\neq i }=\bar{m}=\left( c_2 \frac{x_i}{\rho^N(\boldsymbol{x}) }\frac{1}{N}, 0\right)^{\top}\sim \mathcal{O}\left(\frac1N\right),\\
 & \bar{\alpha}\equiv 0,\ p\equiv0\quad k,i=1,...,N.
\end{align*}
Then we deduce the limiting system for $N\to\infty$:
\begin{footnotesize}
\begin{subequations} \label{example}
\begin{align}
&\partial_t h(t,x,y) +\partial_x h(t,x,y) \Big( c_1 \frac{x}{\int z_1\ g(t,z_1,z_2)\ dz_1dz_2} \\
&\quad\quad + \frac{c_2\ x}{\widehat{\alpha}\ \int z_1\ g(t,z_1,z_2)\ dz_1dz_2}\ \int g(t,z_1,z_2)\ (\partial_{z_1} h(t,z_1,z_2)-\partial_{z_2} h(t,z_1,z_2))\ dz_1dz_2 \Big)\\
&\quad \quad+ \partial_y h(t,x,y)\  r\ y ={l}(t,x,y,g)- \frac{1}{2\widehat{\alpha}} (\partial_x h(t,x,y)-\partial_y h(t,x,y))^2,\\
&\partial_t g(t,x,y) + \partial_x \Big(\Big[  c_1 \frac{x}{\int z_1\ g(t,z_1,z_2)\ dz_1dz_2}+\frac{1}{\widehat{\alpha}} (\partial_x h(t,x,y)-\partial_y h(t,x,y))\\
&\quad\quad +\frac{c_2\ x }{\widehat{\alpha}\ \int z_1\ g(t,z_1,z_2)\ dydx}\  \int  g(t,z_1,z_2)\  (\partial_{z_1} h(t,z_1,z_2)-\partial_{z_2} h(t,z_1,z_2)) \  dz_1\ dz_2  \Big]  g(t,x,y)\Big) \\
& \quad\quad +\partial_y \left( \left[ r\ y+  \frac{1}{\widehat{\alpha}} (\partial_y h(t,x,y)-\partial_x h(t,x,y))    \right]  g(t,x,y)\right)=0.
\end{align}
\end{subequations}
\end{footnotesize}

\end{lemma}

\paragraph{Simplified Model}
We simplify the previous model by considering only the risky portfolio. This is realistic if the prime rate of a national bank is zero. 
We set $y\equiv 0$ in equation \eqref{example} and assume that $h=h(t,x)$ holds.  Then the limit system is given by:

\begin{align*}
&\partial_t h(t,x) +\partial_x h(t,x) \left( c_1 \frac{x}{\int z\ g(t,z)\ dz} +\frac{c_2}{\widehat{\alpha}}  \frac{x}{\int z\ g(t,z)\ dz}\ \int g(t,z)\ \partial_z h(t,z)\ dz \right)\\
&\quad\quad={l}(t,x,g)- \frac{1}{2\widehat{\alpha}} (\partial_x h(t,x))^2,\\
&\partial_t g(t,x) + \partial_x \Big(\Big[  c_1 \frac{x}{\int z\ g(t,z)\ dz}+\frac{1}{\widehat{\alpha}} \partial_x h(t,x)\\
& \quad\quad\quad\quad \quad\quad  + \frac{c_2}{\widehat{\alpha}} \frac{x}{\int z\ g(t,z)\ dz}\ \int g(t,z)\ \partial_z h(t,z)\ dz ] \Big]  g(t,x)\Big) =0.
\end{align*}

\paragraph{Discussion of financial market model}
This microscopic financial market model is regarded as the rational version of  kinetic portfolio optimization model introduced in \cite{Trimborn2019, trimborn2017portfolio}. In this context, rational means that the financial agents solve their optimization problem exactly. The model introduced in \cite{Trimborn2019, trimborn2017portfolio} considers boundedly rational agents in the sense of Simon \cite{simon1982models}. Mathematically, the investors simplify their optimization problem using model predictive control. 
The authors prove that the model can generate well known features of financial markets such as booms and crashes and fat-tails in asset returns. In economic research, there is an ongoing discussion if these phenomena, called stylized facts have their origin in the irrational behavior of market participants.
A detailed analysis of the introduced rational portfolio model might help to answer this question. The discussion of this highly non-linear model is left open for further research.

\section{Conclusion}
We derived the MFG limit system of microscopic dynamics and shown that it is possible to apply MFG theory to an important class of microscopic differential games. In addition, we have seen that the symmetry and the scaling behavior of the microscopic model is crucial in order to obtain the limit.
It is possible to derive different MFG models of the discussed setting. 
Finally, we have shown that financial market models are a prototype candidate for MFG applications. Furthermore, this example motivates the discovery of novel scalings, we have conducted in this study.  \\[2 em] 
Extensions are the generalization of the microscopic differential game model. Such a new setting might include an infinite horizon optimization or stochastic state dynamics.  
Although all results of this study are formal, we believe that this work shows the great applicability of MFG theory to a large class of microscopic systems. We hope that this work clarifies the derivation of complex MFG models and furthermore enables the reader to apply MFG theory to a broad area of applications.

\section*{Acknowledgement}
 M. Herty and T. Trimborn would like to thank the German Research Foundation DFG for the kind support within the Cluster of Excellence Internet of Production (Project-ID: 390621612).
 Funded by the Excellence Initiative of the German federal and state governments.
M. Herty and T.Trimborn acknowlege the support by the ERS Prep Fund - Simulation and Data Science. 
T. Trimborn gratefully acknowledges the support by the Hans-B\"ockler-Stiftung and would like to thank the RWTH Start Up Grant for the kind support. 
\clearpage

		\bibliographystyle{abbrv}	
	\bibliography{literaturmean.bib}

\end{document}